\documentclass[12pt, a4paper]{amsart}

\usepackage[hmargin=30mm, vmargin=25mm, includefoot, twoside]{geometry}
\usepackage[bookmarksopen=true]{hyperref}

\usepackage{amscd}
\usepackage[all,pdf]{xy}

\usepackage{amsfonts,amssymb,verbatim}
\usepackage{latexsym}
\usepackage{mathrsfs}
\usepackage{stmaryrd}
\usepackage{xspace}
\usepackage{enumerate, paralist}
\usepackage{graphicx}
\usepackage[all]{xy}
\usepackage{extarrows}
\usepackage[usenames,dvipsnames]{color}
\usepackage{txfonts, pxfonts}
\usepackage{amsthm}
\usepackage{amsmath}

\newtheorem{thm}{Theorem}[section]
\newtheorem{cor}[thm]{Corollary}
\newtheorem{lem}[thm]{Lemma}
\newtheorem{prop}[thm]{Proposition}

\newtheorem{introthm}{Theorem}

\theoremstyle{definition}
 \newtheorem{introdefn}[introthm]{Definition}
  \newtheorem{introques}[introthm]{Question}

\numberwithin{equation}{section}

\theoremstyle{definition}
\newtheorem{defn}[thm]{Definition}
\theoremstyle{remark}
\newtheorem{rem}[thm]{Remark}

\def\NN{\mathbb{N}}
\def\Nd{\mathcal{N}}
\def\CC{\mathbb{C}}
\def\RR{\mathbb{R}}
\def\AA{\mathbb{A}}
\def\H{\mathcal{H}}
\def\HH{\mathscr{H}}

\def\L{\mathcal{L}}
\def\S{\mathcal{S}}

\def\B{\mathfrak{B}}
\def\K{\mathfrak{K}}
\def\supp{\mathrm{supp}}
\def\diam{\mathrm{diam}}
\def\ppg{\mathrm{prop}}
\def\Ad{\mathrm{Ad}}
\def\Cliff{\mathrm{Cliff}}
\def\Ind{\mathrm{Ind}}
\def\Id{\mathrm{Id}}
\def\d{\mathrm{d}}
\def\Csq{C^*_{sq}}
\def\AAsq{\AA_{sq}}
\def\Asq{A_{sq}}

\begin{document}

\title{Strongly Quasi-local algebras and their $K$-theories}

\author{Hengda Bao}
\author{Xiaoman Chen}
\author{Jiawen Zhang}

\address{School of Mathematical Sciences, Fudan University, 220 Handan Road, Shanghai, 200433, China.}
\email{baohengda@fudan.edu.cn, xchen@fudan.edu.cn, jiawenzhang@fudan.edu.cn}

\thanks{}

\date{}


\baselineskip=16pt

\begin{abstract}
In this paper, we introduce a notion of strongly quasi-local algebras. They are defined for each discrete metric space with bounded geometry, and sit between the Roe algebra and the quasi-local algebra. We show that strongly quasi-local algebras are coarse invariants, hence encoding coarse geometric information of the underlying spaces. We prove that for a discrete metric space with bounded geometry which admits a coarse embedding into a Hilbert space, the inclusion of the Roe algebra into the strongly quasi-local algebra induces an isomorphism in $K$-theory.
\end{abstract}
\date{\today}
\maketitle

\parskip 4pt


\section{Introduction}

Roe algebras are $C^*$-algebras associated to metric spaces which encode coarse geometric information of the underlying spaces. They were introduced by J. Roe in his pioneering work of higher index theory on open manifolds \cite{roe1988:index-thm-on-open-mfds,Roe96}, in which he showed that the $K$-theory of Roe algebras serves as receptacles for indices of elliptic differential operators. Hence the computation of their K-theories becomes a central problem in higher index theory.



An efficient and practical approach is to employ the coarse Baum-Connes conjecture, which asserts that the coarse assembly map from the coarse $K$-homology of the space to the $K$-theory of the Roe algebra is an isomorphism \cite{roe1993coarse,Yu95}. The coarse Baum-Connes conjecture has fruitful and significant applications in geometry and topology, for instance on the Novikov conjecture and the bounded Borel rigidity conjecture (see \emph{e.g.} \cite{guent-tess-yu2012:Borel-stuff,yu1998:Novikov-for-FAD,Yu00}), and on the non-existence of metrics of positive scalar curvature on open Riemannian manifolds (see \emph{e.g.} \cite{schick2014:ICM,yu1997:0-in-the-spec-PSC}).

By definition, the Roe algebra $C^*(X)$ of a discrete metric space $(X,d)$ with bounded geometry is defined to be the norm closure of all locally compact operators $T \in \B(\ell^2(X; \HH))$ (where $\HH$ is an infinite-dimensional separable Hilbert space) with finite propagation in following sense: there exists $R>0$ such that for any $f,g \in \ell^\infty(X)$ acting on $\ell^2(X; \HH)$ by amplified pointwise multiplication, we have $fTg=0$ when their supports are $R$-disjoint (\emph{i.e.}, $d(\supp f, \supp g)>R$). Since general elements in Roe algebras may not have finite propagation, it is usually hard to detect whether a given operator belongs to them or not.

To overcome this issue, J. Roe suggested an asymptotic version of finite propagation called quasi-locality in \cite{roe1988:index-thm-on-open-mfds, Roe96}. More precisely, an operator $T \in \B(\ell^2(X; \HH))$ is quasi-local if for any $\varepsilon>0$ there exists $R>0$ such that for any $f,g \in \ell^\infty(X)$ with $R$-disjoint supports, we have $\|fTg\|<\varepsilon$. We form the \emph{quasi-local algebra}\footnote{Note that a uniform version was already introduced in \cite{intro}.} $C^*_q(X)$ of $X$ as the $C^*$-algebra consisting of all locally compact and quasi-local operators in $\B(\ell^2(X; \HH))$, and show that they are coarse invariants. It is clear that operators with finite propagation are quasi-local, and hence the quasi-local algebra $C^*_q(X)$ contains the Roe algebra $C^*(X)$. 

A natural question is to ask whether these two algebras coincide, which has been extensively studied over the last few decades \cite{engel2015rough,lange1985noethericity,LWZ2018,Roe96,ST19,SZ20}. Currently the most general result is due to \v{S}pakula and the third author \cite{SZ20}, which states that $C^*(X) = C^*_q(X)$ for any discrete metric space with bounded geometry and having Yu's property A. Here property A is a coarse geometric property introduced by Yu \cite{Yu00} in his study on the coarse Baum-Connes conjecture. However, the question remains widely open outside the world of property A \cite{structure,intro}.

On the other hand, the property of quasi-locality is also crucial in the work of Engel \cite{engel2015rough} on index theory of pseudo-differential operators. He discovered that while indices of genuine differential operators on Riemannian manifolds live in the $K$-theory of (appropriate) Roe algebras, the indices of uniform pseudo-differential operators are only known to be in the $K$-theory of quasi-local algebras. 
Hence it is important to study whether the Roe algebra and the quasi-local algebra have the same $K$-theory.

In this paper, we introduce a notion of strong quasi-locality and study associated strongly quasi-local algebras. Our main focus is to study their $K$-theories, which might be a potential approach to attack the higher indices problem above. To illustrate the idea, let us explain when $X$ is uniformly discrete (\emph{i.e.}, there exists $C>0$ such that $d(x,y)>C$ for $x\neq y$). For the general case, see Section \ref{ssec:strong quasi-locality}. Fix an infinite-dimensional separable Hilbert space $\HH$ and denote by $\K(\HH)_1$ the unit ball of the compact operators on $\HH$. We introduce the following:

\begin{introdefn}\label{introdefn:strong quasi-locality}
Let $X$ be a uniformly discrete metric space with bounded geometry and $T \in \B(\ell^2(X; \HH))$. We say that $T$ is \emph{strongly quasi-local} if for any $\varepsilon > 0$ there exists $L >0$ such that for any $L$-Lipschitz map $g:X \rightarrow \K(\HH)_1$, we have
\begin{equation*}
\big\| [T\otimes \Id_{\HH} , \Lambda(g) ] \big\|< \varepsilon
\end{equation*}
where $\Lambda(g) \in \B(\ell^2(X; \HH \otimes \HH))$ is defined by $\Lambda(g)(\delta_x \otimes \xi \otimes \eta):=\delta_x \otimes \xi \otimes g(x)\eta$ for $\delta_x \otimes \xi \otimes \eta \in \ell^2(X; \HH \otimes \HH) \cong \ell^2(X) \otimes \HH \otimes \HH$.
\end{introdefn}

Definition \ref{introdefn:strong quasi-locality} is inspired by a characterisation for quasi-locality provided in \cite{ST19} which states that an operator $T  \in \B(\ell^2(X; \HH))$ is quasi-local \emph{if and only if} for any $\varepsilon > 0$ there exists $L >0$ such that for any $L$-Lipschitz map $g:X \rightarrow \CC$ with $\|g\|_\infty \leq 1$, we have $\|[T,g]\|<\varepsilon$. Hence the notion of strong quasi-locality can be regarded as a compact operator valued version of quasi-locality, and undoubtedly strengthens the original notion (as literally suggested).

Analogous to the case of quasi-locality, we form the \emph{strongly quasi-local algebra} $C^*_{sq}(X)$ as the $C^*$-algebra consisting of all locally compact and strongly quasi-local operators in $\B(\ell^2(X; \HH))$. We show that the strongly quasi-local algebra $C^*_{sq}(X)$ contains the Roe algebra $C^*(X)$, and is contained in the quasi-local algebra $C^*_q(X)$ (see Proposition \ref{prop:relations between Roe, QL and SQL}). We also study coarse geometric features of strongly quasi-local algebras, and show that they are coarse invariants as in the case of Roe algebras and quasi-local algebras (see Corollary \ref{cor:strong coarse invariance}).

Our motivation of introducing strongly quasi-local algebras is that their $K$-theory is relatively easy to handle when the underlying space is coarsely embeddable. More precisely, we prove the following:

\begin{introthm}\label{thm:main result}
	Let $X$ be a discrete metric space of bounded geometry. If $X$ admits a coarse embedding into a Hilbert space, then the inclusion of the associated Roe algebra $C^*(X)$ into the strongly quasi-local algebra $C^*_{sq}(X)$ induces an isomorphism in $K$-theory.
\end{introthm}

Theorem \ref{thm:main result} is the main result of this paper, which is inspired by the well-known theorem of Yu \cite{Yu00} that the coarse Baum-Connes conjecture holds for discrete bounded geometry spaces admitting a coarse embedding into Hilbert space. The proof of Theorem \ref{thm:main result} follows the outline of \cite[Section 12]{willett2020higher} (which originates in \cite{Yu00}), but is more involved and requires new techniques. We divide the proof into several steps, and here let us explain several key ingredients in the proof.

First we prove a coarse Mayer-Vietoris argument for strongly quasi-local algebras (Proposition \ref{prop:M-V sequence for strong quasi-local algebras}), which allows us to cut the space and decompose the associated algebras. Recall that an analogous result for Roe algebras was already established in \cite{HRY93}. This leads to the reduction of the proof for Theorem \ref{thm:main result} to the case of sequences of finite metric spaces with block-diagonal operators thereon (Lemma \ref{lem:reduce to block-diagonal}). 

We would like to highlight a technical lemma used to achieve the coarse Mayer-Vietoris result. Recall that for a quasi-local operator $T \in C^*_q(X)$, it is clear from definition that the restriction $\chi_A T \chi_A$ belongs to $C^*_q(A)$ for any subspace $A$. However, this is not obvious in the case of strongly quasi-local algebras due to certain obstructions on Lipschitz extension (see Remark \ref{rem:obstructions on Lipschitz extension}). To overcome the issue, we provide a characterisation for strong quasi-locality in terms of compact operator valued Higson functions (Proposition \ref{prop:strong quasi-locality iff Higson}). Note that these functions appeared in \cite[Section 4.2]{Wil09} to study the stable Higson corona and the Baum-Connes conjecture. Thanks to the extendability of Higson functions, we obtain a restriction result (Lemma \ref{lem:subspace-wise version}) as required. Moreover by some delicate analysis, we obtain a ``uniform'' version (Proposition \ref{prop:subspace strong quasi-locality}) which plays a key role in following steps.

Then we construct a twisted version of strongly quasi-local algebras (Definition \ref{defn:twisted quasi-local}) for sequences of finite metric spaces, and show that the identity map on the $K$-theory of the strongly quasi-local algebra factors through the $K$-theory of its twisted counterpart (Proposition \ref{prop:strong quasi-local index map isom.}). To achieve, we replace several propagation requirements for twisted Roe algebras by different versions of (strong) quasi-locality, and construct an index map in terms of the Bott-Dirac operators. We would like to point out that for the original quasi-local algebras, there is a technical issue to define the index map (Lemma \ref{lem:commutator}) following the methods either in \cite[Lemma 7.6]{Yu00} or in \cite[Lemma 12.3.9]{willett2020higher}. Hence we have to move to the world of strong quasi-locality.

Finally we prove that the inclusion map from the twisted Roe algebra into the twisted strongly quasi-local algebra induces an isomorphism in $K$-theory (Proposition \ref{prop:iso. of twisted algebras in $K$-theory}). Combining with a diagram-chasing argument, we conclude the proof for Theorem \ref{thm:main result}.

Theorem \ref{thm:main result} should be regarded as a first step to attack the problem whether quasi-local algebras have the same $K$-theory as Roe algebras. More precisely, we pose the following open question:

\begin{introques}
Let $X$ be a metric space with bounded geometry which admits a coarse embedding into Hilbert space. Then do we have $K_\ast(C^*_{sq}(X)) = K_\ast(C^*_q(X))$? 
\end{introques}

The paper is organised as follows: In Section \ref{sec:pre}, we collect notions from coarse geometry and recall the definition of Roe algebras. We also define quasi-local algebras and show that they are coarse invariants. In Section \ref{sec:sq}, we introduce the main concept of this paper---strong quasi-locality and study their permanence property and coarse geometric features. Section \ref{sec:cMV} is devoted to the coarse Mayer-Vietoris sequence for (strongly) quasi-local algebras, based on which we reduce the proof for Theorem \ref{thm:main result} to the case of sequences of finite metric spaces. We introduce twisted strongly quasi-local algebras in Section \ref{sec:twisted}, and construct the index map in Section \ref{sec:index map}. In Section \ref{sec:loc isom}, we show that twisted Roe algebras and twisted strongly quasi-local algebras have the same $K$-theory, and hence conclude the proof in Section \ref{sec:proof of main thm}. The appendix provides a proof for Proposition \ref{prop:Psi function} which slightly strengthens \cite[Proposition 12.1.10]{willett2020higher} and is necessary to achieve the main theorem, hence we give detailed proofs for convenience to readers.

\subsection*{Acknowledgments}
We wish to thank Jinmin Wang and Rufus Willett for several helpful discussions, and Yijun Yao for useful comments after  reading an early version of this paper.

%

\section{Preliminaries}\label{sec:pre}
We start with some notions and definitions. 

\subsection{Notions in coarse geometry}
Here we collect several basic notions.

\begin{defn}\label{defn:basic notions spaces}
	Let $(X,d_X)$ be a metric space, $A\subseteq X$ and $R\geq 0$.
	\begin{enumerate}
		\item $A$ is \emph{bounded} if its diameter $\diam(A):=\sup\{d_X(x,y): x,y\in A\}$ is finite. 
		\item The \emph{$R$-neighbourhood of $A$ in $X$} is $\Nd_R(A):=\{x\in X: d_X(x,A) \leq R\}$. 
       \item $A$ is a \emph{net} in $X$ if there exists some $C>0$ such that $\Nd_C(A)=X$. 
		\item For $x_0\in X$, the open \emph{$R$-ball of $x_0$ in $X$} is $B(x_0;R):=\{x\in X: d_X(x_0,x)<R \}$.
		\item $(X,d_X)$ is said to be \emph{proper} if every closed bounded subset is compact.
		\item If $(X,d_X)$ is discrete, we say that $X$ has \emph{bounded geometry} if for any $r>0$ there exists an $N\in \NN$ such that $|B(x;r)| \leq N$ for any $x\in X$, where $|B(x;r)|$ denotes the cardinality of the set $B(x;r)$. 
	\end{enumerate}
\end{defn}

\begin{defn}\label{defn:basic notions maps}
	Let $f: (X,d_X) \to (Y,d_Y)$ be a map between metric spaces. 
	\begin{enumerate}
		\item $f$ is \emph{uniformly expansive} if there exists a non-decreasing function $\rho_+: [0,\infty) \to [0,\infty)$ such that for any $x,y\in X$, we have:
		\[
		d_Y(f(x),f(y)) \leq \rho_+(d_X(x,y)).
		\]
		\item $f$ is \emph{proper} if for any bounded $B\subseteq Y$, the pre-image $f^{-1}(B)$ is bounded in $X$. 
		\item $f$ is \emph{coarse} if it is uniformly expansive and proper.
		\item $f$ is \emph{effectively proper} if there exists a proper non-decreasing function $\rho_-: [0,\infty) \to [0,\infty)$ such that for any $x,y\in X$, we have:
		\[
		\rho_-(d_X(x,y)) \leq d_Y(f(x),f(y)).
		\]
		\item $f$ is a \emph{coarse embedding} if it is uniformly expansive and effectively proper.
	\end{enumerate}
\end{defn}

Note that $f$ is uniformly expansive is equivalent to that the \emph{expansion function} $\rho_f :[0,\infty)\rightarrow [0,\infty]$ of $f$, defined as
\begin{equation}\label{EQ:expansion function}
\rho_f(s):=\sup\{d_Y(f(x),f(y)) : x,y\in X \mbox{~with~} d_X(x,y)\le s\},
\end{equation}
is finite-valued. 

\begin{defn}\label{defn:basic notions coarse equivalence}
	Let $(X,d_X)$ and $(Y,d_Y)$ be metric spaces. 
	\begin{enumerate}
		\item Two maps $f,g: (X,d_X) \to (Y,d_Y)$ are \emph{close} if the exists $R\ge 0$ such that for all $x\in X$, we have $d_Y(f(x),g(x))\le R$. 
		\item A coarse map $f: (X,d_X) \to (Y,d_Y)$ is called a \emph{coarse equivalence} if there exists another coarse map $g: (Y,d_Y) \to (X,d_X)$ such that $f\circ g$ and $g\circ f$ are close to identities, where $g$ is called a \emph{coarse inverse} to $f$. It is clear that $f$ is a coarse equivalence \emph{if and only if} it is a coarse embedding and $f(X)$ is a net in $Y$. 
		\item $(X,d_X)$ and $(Y,d_Y)$ are said to be \emph{coarsely equivalent} if there exists a coarse equivalence from $X$ to $Y$. 
	\end{enumerate}
\end{defn}

For families of metric spaces and maps, we also need the following notions. 

\begin{defn}\label{defn:basic notions coarse disjoint union}
	Let $\{(X_n,d_{X_n})\}_{n\in \NN}$ be a sequence of finite metric spaces. A \emph{coarse disjoint union} of $\{(X_n,d_{X_n})\}$ is a metric space $(X,d_X)$ where $X$ is the disjoint union of $\{X_n\}$ as a set, and $d_X$ is a metric on $X$ satisfying:
	\begin{itemize}
		\item the restriction of $d_X$ on $X_n$ coincides with $d_{X_n}$;
		\item $d_X(X_n,X\setminus X_n)\rightarrow \infty $ as $n\rightarrow \infty$.
	\end{itemize}
	Note that any two such metric $d_X$ are coarsely equivalent. We say that a sequence $\{(X_n,d_{X_n})\}_{n\in \NN}$ has \emph{uniformly bounded geometry} if its coarse disjoint union has bounded geometry.
\end{defn}

\begin{defn}\label{defn:basic notions uniformly coarse embedding}
	A family of maps $\{f_i: X_i \to Y_i\}_{i\in I}$ between metric spaces is called a \emph{uniformly coarse embedding} if there are non-decreasing proper functions $\rho_\pm: [0,\infty) \to [0,\infty)$ such that 
	\[
	\rho_-(d_{X_i}(x,y)) \leq d_{Y_i}(f_i(x),f_i(y)) \leq \rho_+(d_{X_i}(x,y))
	\]
	for all $i\in I$ and $x,y\in X_i$. We say that \emph{$\{X_i\}_{i\in I}$ uniformly coarsely embeds into Hilbert spaces} if there exists a uniformly coarse embedding $\{f_i: X_i \to E_i\}_{i\in I}$ where each $E_i$ is a Hilbert space.
\end{defn}

It is clear that a sequence of finite metric spaces $\{X_n\}_{n\in \NN}$ uniformly coarsely embeds into Hilbert spaces \emph{if and only if} its coarse disjoint union $\bigsqcup_{n} X_n$ coarsely embeds into some Hilbert space.

\subsection{Roe algebras and quasi-local algebras}
For a proper metric space $(X,d_X)$, recall that an \emph{$X$-module} is a non-degenerate $\ast$-representation $C_0(X) \to \B(\H_X)$ for some infinite-dimensional separable Hilbert space $\H_X$. We also say that $\H_X$ is an $X$-module if the representation is clear from the context. An $X$-module is called \emph{ample} if no non-zero element of $C_0(X)$ acts as a compact operator on $\H_X$. Note that every proper metric space $X$ admits an ample $X$-module.

Let $\H_X$ and $\H_Y$ be ample modules of proper metric spaces $X$ and $Y$, respectively. Given an operator $T\in \B(\H_X,\H_Y)$, the \emph{support} of $T$ is defined to be
\[
\supp(T):=\big\{(y,x)\in Y\times X: \chi_V T\chi_U \neq 0 \mbox{~for~all~neighbourhoods~}U\mbox{~of~}x\mbox{~and~}V\mbox{~of~}y\big\}.
\]
When $X=Y$, the \emph{propagation} of $T \in \B(\H_X)$ is defined to be
\[
\ppg(T):=\sup\{d_X(x,y): (x,y)\in \supp(T)\}.
\]
We say that an operator $T\in \B(\H_X)$ has \emph{finite propagation} if $\ppg(T)$ is finite, and $T$ is \emph{locally compact} if $fT$ and $Tf$ are compact for all $f\in C_0(X)$ (which is equivalent to that both $\chi_K T$ and $T\chi_K$ are compact for all compact subset $K \subseteq X$).

\begin{defn}\label{defn:Roe alg}
	For a proper metric space $X$ and an ample $X$-module $\H_X$, the \emph{algebraic Roe algebra} $\CC[\H_X]$ of $\H_X$ is defined to be the $*$-algebra of locally compact finite propagation operators on $\H_X$, and the \emph{Roe algebra $C^*(\H_X)$} of $\H_X$ is defined to be the norm-closure of $\CC[\H_X]$ in $\B(\H_X)$. 
\end{defn}

It is a standard result that the Roe algebra $C^*(\H_X)$ does not depend on the chosen ample module $\H_X$ up to $*$-isomorphisms, hence denoted by $C^*(X)$ and called the \emph{Roe algebra of $X$}. Furthermore, $C^*(X)$ is a coarse invariant of the metric space $X$ (up to non-canonical $*$-isomorphisms), and their K-theories are coarse invariants up to canonical isomorphisms (see, \emph{e.g.}, \cite{roe1993coarse}).

Now we move on to the case of quasi-locality. 

\begin{defn}\label{defn:quasi-locality}
Given a proper metric space $(X,d_X)$ and an ample $X$-module $\H_X$, an operator $T\in \B(\H_X)$ is said to be \emph{quasi-local} if for any $\varepsilon>0$, there exists $R>0$ such that $T$ has \emph{ $(\varepsilon,R)$-propagation}, \emph{i.e.}, for any Borel sets $A,B\subseteq X$ with $d_X(A,B)\geq R$, we have $\| \chi_A T\chi_B\| < \varepsilon$. 
\end{defn}

It is clear that the set of all locally compact quasi-local operators on $\H_X$ forms a $C^*$-subalgebra of $\B(\H_X)$, which leads to the following:

\begin{defn}\label{defn:quasi-local algebra}
	For a proper metric space $X$ and an ample $X$-module $\H_X$, the set of all locally compact quasi-local operators on $\H_X$ is called the \emph{quasi-local algebra of $\H_X$}, denoted by $C^*_q(\H_X)$. 
\end{defn}

As in the case of Roe algebras, we now show that quasi-local algebras do not depend on the chosen ample modules either.

Let $X$ and $Y$ be proper metric spaces and $\H_X, \H_Y$ be ample modules, respectively. Let $f:X \to Y$ be a coarse map. Recall that a \emph{covering isometry} for $f$ is an isometry $V:\H_X \rightarrow \H_Y$ such that $\supp(V)\subseteq \{(y,x):d_Y(y,f(x))\le C\}$ for some $C\ge 0$. In this case, we also say that $V$ \emph{covers $f$}. It is shown in \cite[Proposition 4.3.4]{willett2020higher} that covering isometries always exist. Following the case of Roe algebras, we have:

\begin{prop}\label{prop:coarse invariance}
	Let $\H_X$ and $\H_Y$ be ample modules for proper metric spaces $X$ and $Y$, respectively. Let $f: X \to Y$ be a coarse map with a covering isometry $V: \H_X\rightarrow \H_Y$. Then $V$ induces the following $\ast$-homomorphism
	\[
	\Ad_V: C^*_q(\H_X)\longrightarrow C^*_q(\H_Y), \quad T\mapsto VTV^*.
	\]
	Furthermore, the induced K-thoeretic map $(\Ad_V)_*: K_*(C^*_q(\H_X)) \to K_*(C^*_q(\H_Y))$ does not depend on the choice of the covering isometry $V$, hence denoted by $f_*$.
\end{prop}

\begin{proof}
Note that there exists a Borel coarse map close to $f$ by \cite[Lemma A.3.12]{willett2020higher}, hence without loss of generality, we can assume that $f$ is Borel and $V$ covers $f$.

Following the same argument as in the Roe case (see, \emph{e.g.}, \cite[Lemma 5.1.12]{willett2020higher}), $VTV^*$ is locally compact. Fix a $t_0>0$ such that $\supp(V)\subseteq\{(y,x) :d_Y(y,f(x))< t_0\}$. For any $\varepsilon>0$, the quasi-locality of $T$ implies that there exists a $R_0>0$ such that $T$ has $(\varepsilon,R_0)$-propagation. We set $R=2t_0 +\rho_f(R_0)+1$ where $\rho_f$ is defined in Equation (\ref{EQ:expansion function}). For any Borel sets $C,D\subseteq Y$ with $d_Y(C,D)\ge R$, it is clear that $d_Y(\Nd_{t_0}(C),\Nd_{t_0}(D))\ge \rho_f(R_0)+1> \rho_f(R_0)$ and hence $d_X(f^{-1}(\Nd_{t_0}(C)),f^{-1}(\Nd_{t_0}(D)))\ge R_0$. Since $V$ covers $f$, we obtain:
	\[
	\chi_CV=\chi_CV\chi_{f^{-1}(\Nd_{t_0}(C))} \quad \mathrm{and} \quad V^*\chi_D=\chi_{f^{-1}(\Nd_{t_0}(D))}V^*\chi_D.
	\]
	Hence
	\[
	\| \chi_C VTV^*\chi_D\| = \| \chi_CV\chi_{f^{-1}(\Nd_{t_0}(C))}T\chi_{f^{-1}(\Nd_{t_0}(D))}V^*\chi_D\| \leq \|\chi_{f^{-1}(\Nd_{t_0}(C))}T\chi_{f^{-1}(\Nd_{t_0}(D))}\| < \varepsilon,
	\]
	which implies that $VTV^*$ is quasi-local.
	
	The second statement follows almost the same argument as in the case of Roe algebra (see, \emph{e.g.}, \cite[Lemma 5.1.12]{willett2020higher}), hence omitted. 
\end{proof}

It is shown in \cite[Proposition 4.3.5]{willett2020higher} that for a coarse equivalence $f: X \to Y$, we can always choose an isometry $V: \H_X \to \H_Y$ covering $f$ such that $V$ is a unitary. Consequently, we obtain the following:

\begin{cor}\label{cor:coarse invariance}
Let $\H_X$ and $\H_Y$ be ample modules for proper metric spaces $X$ and $Y$, respectively. If $X$ and $Y$ are coarsely equivalent, then the quasi-local algebra $C_q^*(\H_X)$ is $\ast$-isomorphic to $C_q^*(\H_Y)$. In particular, for a proper metric space $X$ the quasi-local algebra $C_q^*(\H_X)$ does not depend on the chosen ample $X$-module $\H_X$ up to $\ast$-isomorphisms, hence called \emph{the quasi-local algebra of $X$} and denoted by $C^*_q(X)$.
\end{cor}




\section{Strongly quasi-local algebras}\label{sec:sq}

In this section, we introduce a new class of operator algebras which are called the strongly quasi-local algebras. They sit between Roe algebras and quasi-local algebras and their K-theories will be the main focus of the paper. Here we study their basic properties and coarse geometric features.

Let us begin with some more notions:

\begin{defn}\label{defn:variation}
	Let $(X,d_X), (Y,d_Y)$ be metric spaces and $g:X\rightarrow Y$ be a map.
	\begin{enumerate}
	   \item Given $L>0$, we say that $f$ is \emph{$L$-Lipschitz} if $d_Y(g(x),g(y)) \le Ld_X(x,y)$ for any $x,y\in X$.
		\item Given $A\subseteq X$, $\varepsilon>0$ and $R>0$, we say that $g$ has \emph{$(\varepsilon,R)$-variation on $A$} if for any $x,y\in A$ with $d_X(x,y) < R$, we have $d_Y(g(x),g(y))< \varepsilon$. When $A=X$, we also say that $g$ has \emph{$(\varepsilon,R)$-variation}.
	\end{enumerate}	
\end{defn}

\begin{defn}\label{defn:higson function}
Let $g: X \to \CC$ be a continuous function on a metric space $(X,d_X)$.
\begin{enumerate}
	   \item We say that $g$ is \emph{bounded} if its norm $\|g\|_\infty:=\sup_{x\in X}|g(x)| < \infty$. Denote the set of all bounded continuous functions on $X$ by $C_b(X)$, and by $C_b(X)_1$ the subset consisting of functions with norm at most $1$.
       \item We say that $g$ is a \emph{Higson function} if $g\in C_b(X)$ and for any $\varepsilon>0$ and $R>0$, there exists a compact subset $K\subset X$ such that $g$ has $(\varepsilon,R)$-variation on $X\setminus K$. Denote $C_h(X)$ the set of all Higson functions on $X$.
\end{enumerate}	
\end{defn}

Our notion of strong quasi-locality is inspired by the following result partially from \cite[Theorem 2.8]{ST19}. Recall that for operators $T,S\in \B(\H)$ on some Hilbert space $\H$, their \emph{commutator} is defined to be $[T,S]:=TS-ST$.

\begin{prop}\label{prop:pictures of quasi-locality}
	Let $X$ be a proper metric space, $\H_X$ an ample $X$-module and $T\in\B(\H_X)$ be a locally compact operator. Then the following are equivalent:
	\begin{enumerate}
		\item $T$ is quasi-local in the sense of Definition \ref{defn:quasi-locality};
		\item For any $\varepsilon>0$, there exists $L>0$ such that for any $L$-Lipschitz function $g\in C_b(X)_1$ we have $\|[T,g]\|<\varepsilon$; 
		\item For any $\varepsilon>0$, there exist $\delta,R>0$ such that for any function $g\in C_b(X)_1$ with $(\delta,R)$-variation we have $\|[T,g]\|<\varepsilon$;
		\item $[T,h]$ is a compact operator for any $h\in C_h(X)$.
	\end{enumerate}
\end{prop}

Note that the equivalence among (1), (2) and (4) are the ``easier'' part of \cite[Theorem 2.8]{ST19}. And also note that the equivalence between (1) and (3) can be proved using the same argument therein to show ``(1) $\Leftrightarrow$ (2)'', hence omitted.

\subsection{Strong quasi-locality}\label{ssec:strong quasi-locality}

Now we introduce the notion of strong quasi-locality, where we consider compact operator valued functions instead of complex valued ones in Proposition \ref{prop:pictures of quasi-locality}(3). 

\emph{Throughout the rest of the paper, we only consider proper discrete metric spaces to simplify the notation. We also fix an infinite-dimensional separable Hilbert space $\HH$.}

Let $X$ be a proper discrete metric space and $\H_X$ be an ample $X$-module. For each $x\in X$, denote $\H_x:= \chi_{ \{x\} } \H_X$. An operator $S\in \B(\H_X \otimes \HH)$ can be regarded as an $X$-by-$X$ matrix $(S_{x,y})_{x,y\in X}$, where $S_{x,y} \in \B(\H_{y} \otimes \HH , \H_{x} \otimes \HH)$. Denote $\K(\HH)$ the $C^*$-algebra of compact operators on $\HH$, and $\K(\HH)_1$ its closed unit ball (with respect to the operator norm).

Recall that a map $g:X \rightarrow \K(\HH)$ is \emph{bounded} if $\|g\|_\infty:=\sup_{x\in X}\|g(x)\|<\infty$. Given a bounded map $g:X \rightarrow \K(\HH)$, we define an operator $\Lambda(g)\in \B(\H_X \otimes \HH)$ by setting its matrix entry as follows:
\begin{eqnarray}\label{EQ:Lambda}
\Lambda(g)_{x,y}:=
\begin{cases}
~\Id_{\H_x} \otimes g(x), & y=x; \\ 
~0, & \mbox{otherwise}.
\end{cases}
\end{eqnarray}
Note that this is a block-diagonal operator with respect to the decomposition $\H_X\otimes \HH=\bigoplus_{x\in X} (\H_x\otimes \HH)$. We also write $\Lambda_{\H_X}(g)$ instead of $\Lambda(g)$ when we want to emphasise the module $\H_X$ involved.

The following is the main concept of this paper:

\begin{defn}\label{defn:strongly quasi-local algebra}
Let $X$ be a proper discrete metric space and $\H_X$ be an ample $X$-module. An operator $T\in \B(\H_X)$ is called \emph{strongly quasi-local} if for any $\varepsilon > 0$ there exist $\delta, R >0$ such that for any map $g:X \rightarrow \K(\HH)_1$ with $(\delta, R)$-variation, we have
\begin{equation}\label{EQ:commut strong quasi-local}
\big\| [T\otimes \Id_{\HH} , \Lambda(g) ] \big\|_{ \B (\H_X \otimes \HH)  }< \varepsilon.
\end{equation}
It is easy to see that the set of all locally compact strongly quasi-local operators on $\H_X$ forms a $C^*$-algebra, hence called the \emph{strongly quasi-local algebra of $\H_X$} and denoted by $\Csq(\H_X)$.
\end{defn}

\begin{rem}\label{rem:entry for strong quasi-local}
A direct calculation shows that for $x,y\in X$, the $xy$-matrix entry of the commutator $[T\otimes Id_{\HH} , \Lambda(g) ]$ in Inequality (\ref{EQ:commut strong quasi-local}) is given by:
\begin{equation}\label{EQ:entry for strong quasi-local}
[T\otimes \Id_{\HH} , \Lambda(g) ]_{x,y}=T_{x,y} \otimes \big( g(y)-g(x) \big).
\end{equation}
\end{rem}

The following result records the relation amongst Roe algebras, quasi-local algebras and strongly quasi-local algebras.

\begin{prop}\label{prop:relations between Roe, QL and SQL}
Let $X$ be a proper discrete metric space and $\H_X$ be an ample $X$-module. Then we have:
\begin{enumerate}
  \item $\Csq(\H_X) \subseteq C^*_q(\H_X)$;
  \item If $X$ has bounded geometry, then $C^*(\H_X) \subseteq \Csq(\H_X)$;
  \item If $X$ has bounded geometry and Property A, then $C^*(\H_X) = \Csq(\H_X) = C^*_q(\H_X)$.
\end{enumerate}
\end{prop}

\begin{proof}
 (1). Fix a rank-one projection $p \in \B(\HH)$. For $g\in C_b(X)_1$, we construct $\tilde{g}: X \rightarrow \K(\HH)_1$ by $\tilde{g}(x):=g(x)p$. Since $[T\otimes \Id_{\HH} , \Lambda(\tilde{g})] = [T,g] \otimes p$, the conclusion follows from the definition of strong quasi-locality and Proposition \ref{prop:pictures of quasi-locality}(3).
 
(2). Assume that $T \in \B(\H_X)$ has propagation at most $R$. Then for any $g:X\rightarrow \K(\HH)_1$, the commutator $[T\otimes \Id_{\HH} , \Lambda(g)]$ has propagation at most $R$ from (\ref{EQ:entry for strong quasi-local}). Since $X$ has bounded geometry, it is well-known (see, \emph{e.g.}, \cite[Lemma 12.2.4]{willett2020higher}) that there exists an $N$ depending on $R$ such that for any $g:X\rightarrow \K(\HH)_1$ we have:
\[
\big\|[T\otimes \Id_{\HH} , \Lambda(g)]\big\| \leq N \cdot \sup_{x,y\in X\atop d(x,y)\leq R} \big\|T_{x,y} \otimes \big( g(y)-g(x) \big)\big\|.
\]
This concludes the proof.

(3). It follows from \cite[Theorem 3.3]{SZ20} that $C^*(\H_X) = C^*_q(\H_X)$ under the given assumption, which (together with (1) and (2)) concludes the proof.
\end{proof}

Our next aim is to explore characterisations for strong quasi-locality as in Proposition \ref{prop:pictures of quasi-locality}. First  note that Definition \ref{defn:strongly quasi-local algebra} is a compact operator valued version of condition (3) therein. Unfortunately, we cannot find an appropriate substitute for condition (1) in Proposition \ref{prop:pictures of quasi-locality}. As for condition (2) therein, it is clear that the compact operator valued version is equivalent to strong quasi-locality provided the underlying space is uniformly discrete (\emph{i.e.}, there exists $C>0$ such that $d(x,y)>C$ for $x\neq y$ in $X$). However, it is unclear whether this holds in general. 

As for condition (4) in Proposition \ref{prop:pictures of quasi-locality}, we have the following result concerning compact operator valued Higson functions. 
Recall that a compact operator valued function $h: X \to \K(\HH)$ on a metric space $X$ is a \emph{Higson function} if $h$ is bounded and for any $\varepsilon>0$ and $R>0$, there exists a compact subset $K\subset X$ such that $h$ has $(\varepsilon,R)$-variation on $X\setminus K$.

\begin{prop}\label{prop:strong quasi-locality iff Higson}
Let $X$ be a discrete metric space of bounded geometry and $\H_X$ an ample $X$-module. Then for a locally compact operator $T\in\B(\H_X)$, the following are equivalent:
\begin{enumerate}
 \item $T$ is strongly quasi-local;
 \item $[T\otimes \Id_{\HH} , \Lambda(h)]\in \K(\H_X\otimes\HH)$ for any Higson function $h: X \to \K(\HH)$.
\end{enumerate}
\end{prop}

The proof of Proposition~\ref{prop:strong quasi-locality iff Higson} is almost identical to that of \cite[Theorem 2.8 ``(1) $\Leftrightarrow$ (3)'']{ST19} with minor changes, hence omitted.

\subsection{Strong quasi-locality on subspaces}\label{ssec:Strong quasi-locality on subspaces}

In this subsection, we study the behaviour of strong quasi-locality under taking subspace. First note that in the case of quasi-locality, we have the following observation (which follows directly from Definition \ref{defn:quasi-locality}): given a proper discrete metric space $X$ and an ample module $\H_X$, for any quasi-local operator $T \in \B(\H_X)$ and any $\varepsilon>0$ there exists $R>0$ such that for any $A\subseteq X$ the operator $\chi_A T \chi_A$ has $(\varepsilon,R)$-propagation. In other words, quasi-locality is preserved ``uniformly'' under taking subspaces.

Now we focus on the case of strongly quasi-local operators, and show that they have similar behaviour when taking subspaces. However, the proof is more involved due to the lack of characterisation in terms of $(\varepsilon,R)$-propagation.


\begin{prop}\label{prop:subspace strong quasi-locality}
Let $X$ be a discrete metric space with bounded geometry and $\H_X$ an ample $X$-module. Assume $T\in\B(\H_X)$ is locally compact and strongly quasi-local. Then for any $\varepsilon >0$, there exist $\delta, R>0$ such that for any $A\subseteq X$ and $g: A\rightarrow \K(\HH)_1$ with $(\delta, R)$-variation, we have $\big\| [(\chi_A T\chi_A) \otimes Id_{\HH} , \Lambda(g) ] \big\|_{ \B (\H_A \otimes \HH)  }< \varepsilon$, where $\chi_A T \chi_A$ is naturally regarded as an operator on $\H_A:=\chi_A\H_X$.
\end{prop}

\begin{rem}\label{rem:obstructions on Lipschitz extension}
A natural thought for the proof is to extend a function $g: A\rightarrow \K(\HH)_1$ to $X$ and preserve the variation (or at least with controlled variations). However (as pointed out by Rufus Willett \cite{Wil21}), this is at least as hard as finding extensions with values in a Hilbert space. The problem of extending Hilbert space valued functions is fairly well-studied \cite{Kir34}, and there are known obstructions. In the following, we will bypass the problem using Proposition \ref{prop:strong quasi-locality iff Higson}.
\end{rem}

First we prove a ``subspace-wise'' version of Proposition \ref{prop:subspace strong quasi-locality} (note the difference on orders of quantifiers). To simplify the notation, for $A \subseteq X$ we will regard the characteristic function $\chi_A$ either as the multiplication operator on $\H_X$ or the amplified multiplication operator $\chi_A \otimes \Id_{\HH}$ on $\H_X \otimes \HH$ according to the context. 

\begin{lem}\label{lem:subspace-wise version}
Let $X$ be a discrete metric space with bounded geometry and $\H_X$ an ample $X$-module. Assume that $T\in\B(\H_X)$ is locally compact and strongly quasi-local. Then for any $A\subseteq X$ and $\varepsilon >0$, there exist $\delta, R>0$ such that for any $g: A\rightarrow \K(\HH)_1$ with $(\delta, R)$-variation, we have $\big\| [(\chi_A T\chi_A) \otimes Id_{\HH} , \Lambda(g) ] \big\|_{ \B (\H_A \otimes \HH)  }< \varepsilon$.
\end{lem}

\begin{proof}
By Proposition \ref{prop:strong quasi-locality iff Higson}, we know that $[T\otimes \Id_{\HH} , \Lambda(h)]\in \K(\H_X\otimes\HH)$ for any Higson function $h: X \to \K(\HH)$. Now fix a subspace $A \subseteq X$. For any Higson function $g: A \to \K(\HH)$, it follows from \cite[Lemma 4.3.4]{Wil09} that $g$ can be extended to a Higson function $\tilde{g}: X \to \K(\HH)$. Hence we obtain:
\[
[(\chi_A T\chi_A) \otimes Id_{\HH} , \Lambda(g) ]= \chi_A [T \otimes Id_{\HH} , \Lambda(\tilde{g}) ]\chi_A \in \chi_A  \K(\H_X\otimes\HH) \chi_A \subseteq  \K(\H_A\otimes\HH).
\]
Using Proposition \ref{prop:strong quasi-locality iff Higson} again, we obtain that $\chi_A T\chi_A$ is strongly quasi-local on $\H_A$. This concludes the proof.
\end{proof}

%

\begin{proof}[Proof of Proposition~\ref{prop:subspace strong quasi-locality}]
Fix a base point $x_0 \in X$, and write $B_S$ for $B(x_0;S)$ where $S>0$. Assume the contrary, then there exists some $\varepsilon_0>0$ such that for each $n\in\NN$, there exist $A_n\subseteq X$ and $g_n:A_n\rightarrow \K(\HH)_1$ with $(\frac{1}{n},n)$-variation on $A_n$ such that
\begin{equation}\label{EQ:subspace1}
\big\|[\chi_{A_n}T\chi_{A_n}\otimes \Id_{\HH}, \Lambda(g_n)] \big\|>\varepsilon_0.
\end{equation}
Without loss of generality, we can assume that each $A_n$ is finite.
For the above $\varepsilon_0$, there exists $R_0>0$ such that $T$ has $(\frac{\varepsilon_0}{8},R_0)$-propagation.

\emph{Claim.} For any $R > R_0$, there exists $N\in\NN$ such that for any $n\ge N$ we have:
\[
\big\|[\chi_{A_n\setminus B_R}T\chi_{A_n\setminus B_R}\otimes \Id_{\HH} , \Lambda(g_n)]\big\| > \frac{\varepsilon_0}{8}.
\]

We assume the contrary, \emph{i.e.}, assume that there exist some $R > R_0$ and an increasing sequence $(n_k)_{k=1}^\infty \subseteq \NN$ tending to infinity such that
\[
\big\|[\chi_{A_{n_k}\setminus B_{R}} T\chi_{A_{n_k}\setminus B_{R}}\otimes \Id_{\HH} ,\Lambda(g_{n_k})]\big\| \le \frac{\varepsilon_0}{8}.
\]
Since $d_X(B_{R}, X\setminus B_{2R}) \ge R  >R_0$ we obtain:
\[
\big\|\chi_{A_{n_k}\cap B_{R}}T\chi_{A_{n_k}\setminus B_{2R}}\big\| \le \frac{\varepsilon_0}{8} \quad \mbox{and} \quad \big\|\chi_{A_{n_k}\setminus B_{2R}}T \chi_{A_{n_k}\cap B_{R}}\big\| \le \frac{\varepsilon_0}{8}.
\]
Now we cut up the operator $\chi_{A_{n_k}}T\chi_{A_{n_k}}$ as follows:
\begin{eqnarray*}
		\chi_{A_{n_k}}T\chi_{A_{n_k}} &=&
		\chi_{A_{n_k}\cap B_{2R}}T\chi_{A_{n_k}\cap B_{2R}}
		+ \chi_{A_{n_k}\setminus B_{R}} T\chi_{A_{n_k}\setminus B_{2R}} \\
		& & + \chi_{A_{n_k}\setminus B_{2R}} T\chi_{A_{n_k}\cap (B_{2R}\setminus B_{R})}
		+\chi_{A_{n_k}\cap B_{R}}T\chi_{A_{n_k}\setminus B_{2R}}
		+\chi_{A_{n_k}\setminus B_{2R}}T \chi_{A_{n_k}\cap B_{R}}.
\end{eqnarray*}
Combining the above inequalities with (\ref{EQ:subspace1}), we obtain:
\[
\big\|[\chi_{A_{n_k}\cap B_{2R}}T\chi_{A_{n_k}\cap B_{2R}} \otimes \Id_{\HH} , \Lambda(g_{n_k})]\big\| > \varepsilon_0 - 2 \cdot \frac{\varepsilon_0}{8}- 2 \cdot \frac{\varepsilon_0}{4} =\frac{\varepsilon_0}{4},
\]
which is a contradiction since $A_{n_k}\cap B_{2R}$ is contained in a fixed finite subset $B_{2R}$ and $g_{n_k}$ has $(\frac{1}{n_k},n_k)$-variation on $A_{n_k}\cap B_{2R}$. Hence we prove the Claim.

Now we continue the proof of Proposition~\ref{prop:subspace strong quasi-locality}. Set $\tilde{A}_1:=A_1, n_1:=1$ and choose $R_1>R_0$ such that $\tilde{A}_1 \subseteq B_{R_1-2}$. We recursively choose subsets $\tilde{A}_1, \tilde{A}_2, \ldots$, positive numbers $R_1<R_2<\cdots$ and natural numbers $n_1<n_2<\cdots$ as follows. Suppose that $\tilde{A}_1, \ldots, \tilde{A}_{i-1}$, $R_1< \cdots<R_{i-1}$ and $n_1<\cdots<n_{i-1}$ are chosen for $i\ge 2$. The Claim implies that there exists a natural number $n_i>n_{i-1}$ such that 
\[
\big\|[\chi_{A_{n_i}\setminus B_{R_{i-1}}}T\chi_{A_{n_i}\setminus B_{R_{i-1}}}\otimes \Id_{\HH} , \Lambda(g_{n_i})]\big\| > \frac{\varepsilon_0}{8}.
\]
We take 	$\tilde{A}_i:=A_{n_i}\setminus B_{R_{i-1}}$ (which is non-empty by the above estimate) and choose $R_i>R_{i-1}$ such that $\tilde{A}_1\sqcup\cdots\sqcup\tilde{A}_i \subseteq B_{R_i-2^i}$. In summary, we obtain non-empty subsets $\{\tilde{A}_i\}_{i\in \NN}$ and functions $\hat{g_i}:=g_{n_i}|_{\tilde{A}_i}: \tilde{A}_i \to \K(\HH)_1$ with $(\frac{1}{n_i},n_i)$-variation such that 
\[
\big\|[\chi_{\tilde{A}_i}T\chi_{\tilde{A}_i}\otimes \Id_{\HH} , \Lambda(\hat{g_i})]\big\| > \frac{\varepsilon_0}{8}.
\]

Define $A:=\bigsqcup_{i\in \NN} \tilde{A}_i$ and extend each $\hat{g_i}$ to $A$ by zero on the complement (still denoted by $\hat{g_i}$). It is clear from the above construction that $d_X(\tilde{A}_i , A\setminus \tilde{A}_i)\ge 2^{i-1}$, and hence $\hat{g_i}$ has $(\frac{1}{i},i)$-variation on $A$. Moreover, we have:
\[
\big\|[\chi_AT\chi_A\otimes \Id_{\HH} , \Lambda(\hat{g_i})]\big\| > \frac{\varepsilon_0}{8}.
\]
This is a contradiction to Lemma \ref{lem:subspace-wise version}. Hence we conclude the proof.
\end{proof}

\subsection{Coarse invariance of strongly quasi-local algebras}

In this subsection, we show that strongly quasi-local algebras are coarse invariants provided the underlying spaces have bounded geometry. In particular, this implies that strongly quasi-local algebras are independent of ample modules. The proof follows the outline of that for Proposition \ref{prop:coarse invariance} but is more involved.

\begin{prop}\label{prop:strong coarse invariance}
Let $X,Y$ be discrete metric spaces with bounded geometry and $\H_X, \H_Y$ be ample modules for $X$ and $Y$, respectively.  Let $f: X \to Y$ be a coarse map with a covering isometry $V: \H_X\rightarrow \H_Y$. Then $V$ induces the following $\ast$-homomorphism
	\[
	\Ad_V: \Csq(\H_X)\longrightarrow \Csq(\H_Y), T\mapsto VTV^*.
	\]
Furthermore, the induced K-thoeretic map $(\Ad_V)_*: K_*(\Csq(\H_X)) \to K_*(\Csq(\H_Y))$ does not depend on the choice of the covering isometry $V$, hence denoted by $f_*$.
\end{prop}

\begin{proof}
We only show that $VTV^* \in \Csq(\H_Y)$ if $T\in \Csq(\H_X)$. The ``Furthermore'' part follows almost the same argument as in the case of Roe algebra, hence omitted. 

First note that $VTV^*$ is locally compact as in Proposition~\ref{prop:coarse invariance}. To see that $VTV^*$ is strongly quasi-local, we assume that $\supp (V) \subseteq \{(y,x):d_Y(f(x),y)<R_0\}$ for some $R_0>0$. Since $Y$ has bounded geometry, there exists $N\in \NN$ such that $\big|\{y\in Y:d_Y(f(x),y)<R_0\}\big|\le N$ for any $x\in X$. Hence we can write:
\[
V=W_1 + W_2 + \cdots +W_N
\]
where each $W_i\in \B(\H_X,\H_Y)$ satisfies $\supp(W_i) \subseteq \supp (V)$, $\supp(W_i) \cap \supp(W_j)$ is empty for any $j\neq i$, and for any pair $(y_1,x_1)\neq (y_2,x_2) \in \supp(W_i)$ we have $x_1 \ne x_2$. Set $M:=\max\big\{\|W_i\|:i=1,\ldots,N\big\}$. For later use, we denote $D_i:=\{x\in X: \exists y\in Y \mathrm{\ such\ that\ } (y,x)\in \supp(W_i)\} \subseteq X$. It follows that for each $i$ there exists a map $t_i:D_i \to Y$ such that $(y,x)\in \supp(W_i)$ if and only if $x\in D_i$ and $y=t_i(x)$.

It suffices to show that each $W_iTW_j^*$ is strongly quasi-local. Given an $\varepsilon > 0$, there exist $\delta', R'>0$ such that for any $\varphi : X\rightarrow \K(\HH)_1$ with $(\delta', R')$-variation, we have $\|[T\otimes \Id_{\HH},\Lambda(\varphi)]\|<\frac{\varepsilon}{2M^2}$. Set 
\[
\delta=\min\big\{\frac{\varepsilon}{4M^2\|T\|},\delta'\big\} \quad \mbox{and} \quad R=R_0+\rho_f(R'),
\]
where $\rho_f$ is defined in (\ref{EQ:expansion function}). For any $g:Y\rightarrow \K(\HH)_1$ with $(\delta,R)$-variation and each $i$, we construct $\varphi_i: X\rightarrow \K(\HH)_1$ as follows:
	\[
	\varphi_i(x):=
	\begin{cases}
	~(g\circ t_i)(x), & \mbox{if~} x\in D_i; \\
	~0, & \mbox{otherwise}.
	\end{cases}
	\]
It is clear that $(t_i(x),x) \in \supp(W_i) \subseteq \supp(V)\subseteq \{(y,x):d_Y(f(x),y)<R_0\}$ for each $i$ and $x\in D_i$, which implies that $d_Y\big( t_i(x),f(x) \big)<R_0 \leq R$. Hence we obtain
\[
\sup_{x\in D_i} \big\|\varphi_i(x) - (g\circ f)(x)\big\| \leq \delta,
\]
which implies that for each $i$ we have:
\begin{equation}\label{EQ:coarse equivalence}
\|\Lambda(\varphi_i - g\circ f) (W_i^*\otimes \Id_{\HH})\| \leq \delta M \quad \mbox{and} \quad \|(W_i\otimes \Id_{\HH}) \Lambda(\varphi_i - g\circ f)\| \leq \delta M.
\end{equation}

On the other hand, direct calculations show that for each $i$ we have:
\[
\Lambda(g)(W_i \otimes \Id_{\HH})=(W_i\otimes \Id_{\HH})\Lambda(\varphi_i)\quad \mbox{and}\quad (W_i^* \otimes \Id_{\HH})\Lambda(g) = \Lambda(\varphi_i)(W_i^*\otimes \Id_{\HH}).
\]
Hence we obtain:
\begin{eqnarray*}
&&\big\|[(W_iTW_j^*)\otimes \Id_{\HH}, \Lambda(g)] \big\| \\
&=& \big\|\big((W_iT)\otimes \Id_{\HH}\big) \Lambda(\varphi_j) (W_j^*\otimes \Id_{\HH}) - (W_i\otimes \Id_{\HH}) \Lambda(\varphi_i) \big((TW_j^*)\otimes \Id_{\HH}\big)\big\| \\
&\leq & \big\|\big((W_iT)\otimes \Id_{\HH}\big) \Lambda(g\circ f) (W_j^*\otimes \Id_{\HH}) - (W_i\otimes \Id_{\HH}) \Lambda(g\circ f) \big((TW_j^*)\otimes \Id_{\HH}\big)\big\| + 2M^2\|T\|\delta\\
&\leq & \big\|(W_i\otimes \Id_{\HH}) [T\otimes \Id_{\HH}, \Lambda(g\circ f)](W_j^*\otimes \Id_{\HH})\big\| + \frac{\varepsilon}{2},
\end{eqnarray*}
where we use (\ref{EQ:coarse equivalence}) in the second inequality. Note that $g\circ f : X\rightarrow \K(\HH)_1$ has $(\delta',R')$-variation. Hence $\|[T\otimes \Id_{\HH}, \Lambda(g\circ f)]\| < \frac{\varepsilon}{2M^2}$, which implies:
\[
\big\|[(W_iTW_j^*)\otimes \Id_{\HH}, \Lambda(g)] \big\| < M^2 \cdot \frac{\varepsilon}{2M^2} + \frac{\varepsilon}{2} = \varepsilon.
\]
Hence each $W_iTW_j^*$ is strongly quasi-local.
\end{proof}

As a direct corollary, we obtain:

\begin{cor}\label{cor:strong coarse invariance}
Let $\H_X$ and $\H_Y$ be ample modules for discrete metric spaces $X$ and $Y$ of bounded geometry, respectively. If $X$ and $Y$ are coarsely equivalent, then the strongly quasi-local algebra $\Csq(\H_X)$ is $\ast$-isomorphic to $\Csq(\H_Y)$. In particular, for a discrete metric space $X$ of bounded geometry the strongly quasi-local algebra $\Csq(\H_X)$ does not depend on the chosen ample $X$-module $\H_X$ up to $\ast$-isomorphisms, hence called \emph{the strongly quasi-local algebra of $X$} and denoted by $\Csq(X)$.
\end{cor}


\subsection{The case for sequences of metric spaces}

Here we study the strongly quasi-local algebra for a sequence of metric spaces. This is crucial to analyse the ``building blocks'' when we prove our main theorem.

Let $\{X_n\}_{n\in \NN}$ be a sequence of finite metric spaces and $\rho_n: C_0(X_n) \to \B(\H_n)$ an ample module for $X_n$. Let $X$ be a coarse disjoint union of $\{X_n\}$ and $\H_X:=\bigoplus_n \H_n$. Since $C_0(X) = \bigoplus_n C_0(X_n)$, we can compose $\rho_n$ into a single representation:
\[
\rho=\bigoplus_n \rho_n: C_0(X) \to \prod_n \B(\H_n) \subseteq \B(\H_X).
\]
It is clear that $\rho$ is an ample module for $X$. In the following, we also regard a sequence $(T_n)_{n\in \NN} \in \prod_n \B(\H_n)$ as a single operator in $\B(\H_X)$.

For a locally compact operator $T\in \B(\H_X)$ with finite propagation, it follows directly from definition that $T$ is block-diagonal upto compact operators. Hence we have the following decomposition for Roe algebras:
\begin{lem}\label{lem:cdu to separated case I}
Using the same notation as above, we have:
	\begin{enumerate}
		\item $\big(C^*(\H_X) \cap \prod_n \B(\H_n)\big) + \K(\H_X) = C^*(\H_X)$;
		\item $\big(C^*(\H_X) \cap \prod_n \B(\H_n)\big) \cap \K(\H_X) = \bigoplus_n C^*(\H_n)$.
	\end{enumerate}
\end{lem}

In the case of (strong) quasi-locality, we have similar results as follows. We only need those concerning strong quasi-locality for later use, while we collect them here for completion.

\begin{lem}\label{lem:cdu to separated case II}
	Using the same notation as above, we have:
	\begin{enumerate}
		\item $\big(C^*_q(\H_X) \cap \prod_n \B(\H_n)\big) + \K(\H_X) = C^*_q(\H_X)$;
		\item $\big(C^*_q(\H_X) \cap \prod_n \B(\H_n)\big) \cap \K(\H_X) = \bigoplus_n C^*_q(\H_n)$;
		\item $\big(\Csq(\H_X) \cap \prod_n \B(\H_n)\big) + \K(\H_X) = \Csq(\H_X)$;
		\item $\big(\Csq(\H_X) \cap \prod_n \B(\H_n)\big) \cap \K(\H_X) = \bigoplus_n \Csq(\H_n)$.
	\end{enumerate}
\end{lem}

\begin{proof}

The proof is different from that for Roe algebras, and we only prove (3) and (4) since the other two are similar and easier.


For (3): note that $\K(\H_X) \subseteq \Csq(\H_X)$, hence the left hand side is contained in the right one. For the converse, it follows from \cite[Corollary 4.3]{ST19} that for any $T\in\Csq(\H_X) \subseteq C^*_q(\H_X)$ and $\varepsilon >0$, there exists some $N\in \NN$ such that 
\[
\big\| T - \big(\sum_{i=1}^N \chi_{X_i}\big) T \big(\sum_{i=1}^N \chi_{X_i}\big) - \sum_{i>N} \chi_{X_i} T \chi_{X_i} \big\| < \varepsilon.
\]
Since $T$ is locally compact, then $(\sum_{i=1}^N \chi_{X_i}) T (\sum_{i=1}^N \chi_{X_i})$ is compact. It suffice to show that $\sum_{i>N} \chi_{X_i} T \chi_{X_i} \in \Csq(\H_X)$. Given $\varepsilon>0$, the strong quasi-locality of $T$ implies that there exist $\delta,R>0$ such that for any $g: X \to \K(\HH)_1$ with $(\delta, R)$-variation, we have $\big\| [T \otimes \Id_{\HH} , \Lambda(g) ] \big\|< \varepsilon$. Now for any such $g$, we have:
\begin{eqnarray*}
\big\| \big[\big(\sum_{i>N} \chi_{X_i} T \chi_{X_i}\big)\otimes \Id_{\HH}, \Lambda(g) \big] \big\| &=& \sup_{i>N} \big\| [(\chi_{X_i} T\chi_{X_i}) \otimes \Id_{\HH} , \Lambda(g) ] \big\| \\
&=& \sup_{i>N} \big\|\chi_{X_i} [ T \otimes \Id_{\HH} , \Lambda(g) ] \chi_{X_i}\big\| < \varepsilon.
\end{eqnarray*}
Hence we obtain that $\sum_{i>N} \chi_{X_i} T \chi_{X_i}$ is strongly quasi-local, which concludes (3).

For (4): note that $C^*(\H_n)= C^*_q(\H_n)= \Csq(\H_n)= \K(\H_n)$ for each $n$ and hence:
\begin{align*}
\big(\Csq & (\H_X) \cap \prod_n \B(\H_n)\big) \cap \K(\H_X) = \Csq(\H_X) \cap \big( \prod_n \B(\H_n) \cap \K(\H_X) \big) \\
&= \Csq(\H_X) \cap \bigoplus_n \K(\H_n) = \bigoplus_n \K(\H_n) = \bigoplus_n \Csq(\H_n).
\end{align*}
Hence we conclude the proof.
\end{proof}

For later use, we introduce the following notion of (strong) quasi-locality for a sequence of operators. Note that the definition is nothing but uniform versions of (strong) quasi-locality.
\begin{defn}\label{defn:sql for sequence}
Let $\{X_n\}_{n\in \NN}$ be a sequence of finite metric spaces and $\rho_n: C_0(X_n) \to \B(\H_n)$ be ample modules. For a sequence $(T_n)_{n\in \NN}$ where $T_n \in \B(\H_n)$, we say that:
\begin{enumerate}
 \item $(T_n)_{n\in \NN}$ is \emph{uniformly quasi-local} if for any $\varepsilon>0$ there exists $R>0$ such that for any $n\in \NN$ and $C_n, D_n \subseteq X_n$ with $d(C_n, D_n)\le R$, we have $\|\chi_{C_n} T_n \chi_{D_n}\|<\varepsilon$.
 \item $(T_n)_{n\in \NN}$ is \emph{uniformly strongly quasi-local} if for any $\varepsilon>0$ there exist $\delta,R >0$ such that for any $n \in \NN$ and $g_n:X_n\rightarrow \K(\HH)_1$ with $(\delta, R)$-variation, we have $\| [T_n \otimes Id_{\HH} , \Lambda(g_n) ] \| < \varepsilon$.
\end{enumerate}
\end{defn}


\begin{lem}\label{lem:char for sequence of strong quasi-locality}
Let $\{X_n\}_{n\in \NN}$ be a sequence of finite metric spaces, $\rho_n: C_0(X_n) \to \B(\H_n)$ be ample modules and $\H_X:=\bigoplus_n \H_n$. For a sequence $(T_n)_{n\in \NN} \in \prod_n \K(\H_n)$, we have:
\begin{enumerate}
 \item $(T_n)\in C^*_q(\H_X)$ \emph{if and only if} $(T_n)$ is uniformly quasi-local.
 \item $(T_n)\in \Csq(\H_X)$ \emph{if and only if} $(T_n)$ is uniformly strongly quasi-local.
\end{enumerate}
Hence if $(T_n)$ is uniformly strongly quasi-local then it is uniformly quasi-local.
\end{lem}

The proof is straightforward
, hence omitted. 

Analogous to the coarse invariance of Roe algebras, we have the following result concerning sequences of spaces. The proof is similar, hence omitted.

\begin{prop}\label{prop:independent of modules}
Let $\{X_n\}_{n\in \NN}$ be a sequence of finite metric spaces with uniformly bounded geometry, and $\rho_n: C_0(X_n) \to \B(\H_n)$ be an ample module for $X_n$. Let $X$ be a coarse disjoint union of $\{X_n\}$ and $\H_X:=\bigoplus_n \H_n$. Then the K-theories $K_\ast\big( C^*(\H_X) \cap \prod_n \B(\H_n) \big), K_\ast\big( C^*_q(\H_X) \cap \prod_n \B(\H_n) \big)$ and $K_\ast\big( \Csq(\H_X) \cap \prod_n \B(\H_n) \big)$ are independent of $\rho_n$ up to canonical isomorphisms. 
\end{prop}

\section{The coarse Mayer-Vietoris sequence}\label{sec:cMV}

The tool of Mayer-Vietories sequences is widely used within different area of mathematics, especially in algebraic topology. It provides a ``cutting and pasting'' procedure, which allows us to obtain global information from local pieces.

In coarse geometry, Higson, Roe and Yu introduced a coarse Mayer-Vietoris sequence for K-theories of Roe algebras associated to a suitable decomposition of the underlying metric space in \cite{HRY93}. More precisely, recall that a closed cover $(A,B)$ of a metric space $X$ is said to be \emph{$\omega$-excisive} if for each $r>0$ there is some $s>0$ such that $\Nd_r(A)\cap \Nd_r(B) \subseteq \Nd_s(A\cap B)$. Associated to an $\omega$-excisive closed cover $(A,B)$ of a metric space $X$, we have the following short exact sequence (called the \emph{coarse Mayer-Vietoris sequence}):
\[
\begin{CD}
K_0(C^*(A\cap B)) @>>> K_0(C^*(A))\oplus K_0(C^*(B)) @>>> K_0(C^*(X)) \\
@AAA                            & &                    @VVV \\
K_1(C^*(X)) @<<< K_1(C^*(A))\oplus K_1(C^*(B)) @<<< K_1(C^*(A\cap B)).
\end{CD}
\]

In this section, we explore a coarse Mayer-Vietoris sequence for strongly quasi-local algebras and use it to reduce the proof of Theorem \ref{thm:main result} to the case of ``sparse'' spaces. Let $X$ be a discrete metric space with bounded geometry and $\H_X$ be an ample $X$-module.

\begin{defn}\label{defn:strong M-V ideals}
	Let $A$ be a (closed) subset of $X$. Denote by $\Csq(A,X)$ the norm-closure of the set of all operators $T\in \Csq(\H_X)$ with support contained in $\Nd_R(A) \times \Nd_R(A)$ for some $R\geq 0$. 
\end{defn}



\begin{lem}\label{lem:strong M-V ideals}
$\Csq(A,X)$ is a closed two-sided $\ast$-ideal in $\Csq(\H_X)$.
\end{lem}

\begin{proof}
	It suffices to show that for $T,S \in \Csq(\H_X)$ with $\supp(T) \subseteq \Nd_R(A) \times \Nd_R(A)$ for some $R\geq 0$, then $TS$ and $ST$ belong to $\Csq(A,X)$. By Proposition \ref{prop:relations between Roe, QL and SQL}(1), we know that $S \in C^*_q(\H_X)$. Hence for any $\varepsilon>0$, there exists $R_0>0$ such that $S$ has $(\frac{\varepsilon}{\|T\|},R_0)$-propagation. It follows that
\begin{equation*}
	\|TS - \chi_{\Nd_R(A)}TS\chi_{\Nd_{R+R_0}(A)}\| = \|\chi_{\Nd_R(A)}T(\chi_{\Nd_R(A)}S - \chi_{\Nd_R(A)}S\chi_{\Nd_{R+R_0}(A)})\|  < \varepsilon.
\end{equation*}
Hence by definition, we obtain that $TS \in \Csq(A,X)$. A similar argument shows that $ST \in \Csq(A,X)$ as well, which concludes the proof.
\end{proof}

Based on a similar argument as in the proof of \cite[Section 5/Lemma 1]{HRY93} together with Corollary \ref{cor:strong coarse invariance}, we have the following:


\begin{lem}\label{lem:strong M-V ideals have same K-theory}
For a (closed) subset $A\subseteq X$, take an isometry $V$ covering the inclusion $i: A \hookrightarrow X$. Then the range of $\Ad_V:\Csq(A) \to \Csq(X)$ is contained in $\Csq(A,X)$. Furthermore, the map $i_*:K_*(\Csq(A)) \to K_*(\Csq(A,X))$ is an isomorphism.
\end{lem}


We also have the following result analogous to \cite[Section 5/Lemma 2]{HRY93}:

\begin{lem}\label{lem:strong M-V ideal calculations}
Let $(A,B)$ be an $\omega$-excisive (closed) cover of $X$, then we have
	\[
	\Csq(A,X)+\Csq(B,X)=\Csq(X)
	\]
	and
	\[
	\Csq(A,X)\cap \Csq(B,X)= \Csq(A\cap B,X).
	\]
\end{lem}

\begin{proof}
Given $T\in \Csq(X)$ and $\varepsilon>0$, it follows from Proposition \ref{prop:relations between Roe, QL and SQL}(1) that there exists $R>0$ such that $T$ has $(\varepsilon,R)$-propagation. Note that $T=\chi_A T + \chi_{B \setminus A} T$ since $A \cup B=X$, then $T$ is $2\varepsilon$-close to $\chi_A T \chi_{\Nd_R(A)} + \chi_{B \setminus A} T \chi_{\Nd_R(B \setminus A)}$. Hence we obtain that $\Csq(A,X)+\Csq(B,X)$ is dense in $\Csq(X)$. It follows from a standard argument in $C^*$-algebras (\emph{e.g.}, \cite[Section 3/Lemma 1]{HRY93}) that $\Csq(A,X)+\Csq(B,X)=\Csq(X)$.

Concerning the second equation, we only need to show that $\Csq(A,X) \Csq(B,X) \subseteq \Csq(A\cap B,X)$. Fix $T, S \in \Csq(X)$ with $\supp(T) \subseteq \Nd_R(A) \times \Nd_R(A)$ and $\supp(S) \subseteq \Nd_{R}(B) \times \Nd_{R}(B)$ for some $R>0$. 
The assumption of $\omega$-excision implies that there exists an $L>0$ such that $\Nd_R(A) \cap \Nd_{R}(B) \subseteq \Nd_L (A \cap B)$. Hence we have $TS = T \chi_{\Nd_L (A \cap B)} S$. For any $\varepsilon>0$ there exists an $L'>0$ such that $T$ has $(\frac{\varepsilon}{\|S\|},L')$-propagation and $S$ has $(\frac{\varepsilon}{\|T\|},L')$-propagation. Hence we have:
\[
	\|TS - \chi_{\Nd_{L+L'} (A \cap B)} TS \chi_{\Nd_{L+L'} (A \cap B)}\|  \leq 2\varepsilon.
\]
Therefore we obtain that $TS \in \Csq(A\cap B,X)$, which concludes the proof.
\end{proof}

Applying the Mayer-Vietoris sequence in $K$-theory for $C^*$-algebras (see \cite[Section 3]{HRY93}) to the ideals $\Csq(A,X), \Csq(B,X)$ in $\Csq(X)$ and combining with Lemma \ref{lem:strong M-V ideals have same K-theory} and Lemma \ref{lem:strong M-V ideal calculations}, we obtain the following coarse Mayer-Vietoris principle for strongly quasi-local algebras:

\begin{prop}\label{prop:M-V sequence for strong quasi-local algebras}
	Let $(A,B)$ be a (closed) $\omega$-excisive cover of $X$. Then there is a six-term exact sequence
	\[
	\begin{CD}
	K_0(\Csq(A\cap B)) @>>> K_0(\Csq(A))\oplus K_0(\Csq(B)) @>>> K_0(\Csq(X)) \\
	@AAA                               & &                    @VVV \\
	K_1(\Csq(X)) @<<< K_1(\Csq(A))\oplus K_1(\Csq(B)) @<<< K_1(\Csq(A\cap B)).
	\end{CD}
	\]
\end{prop}

For future use, we record that the same argument can be applied to obtain the Mayer-Vietoris principle for quasi-local algebras as follows. However, this will not be used in this paper.

\begin{prop}\label{prop:M-V sequence for quasi-local algebras}
	Let $(A,B)$ be a (closed) $\omega$-excisive cover of $X$. Then there is a six-term exact sequence
	\[
	\begin{CD}
	K_0(C^*_q(A\cap B)) @>>> K_0(C^*_q(A))\oplus K_0(C^*_q(B)) @>>> K_0(C^*_q(X)) \\
	@AAA                               & &                    @VVV \\
	K_1(C^*_q(X)) @<<< K_1(C^*_q(A))\oplus K_1(C^*_q(B)) @<<< K_1(C^*_q(A\cap B)).
	\end{CD}
	\]
\end{prop}


%

%
%

Now we use Proposition \ref{prop:M-V sequence for strong quasi-local algebras} to reduce the proof of Thereom~\ref{thm:main result} to the case of block-diagonal operators:

\begin{lem}\label{lem:reduce to block-diagonal}
	To prove Theorem~\ref{thm:main result} for all bounded geometry metric spaces that coarsely embed into Hilbert space, it suffices to prove that for any sequence of finite metric spaces $\{Y_n\}_{n=1}^\infty$ which has uniformly bounded geometry and uniformly coarsely embeds into Hilbert space, the inclusion $C^*(\H_Y) \cap \prod_n \B(\H_n) \hookrightarrow \Csq(\H_Y) \cap \prod_n \B(\H_n)$ induces isomorphisms in $K$-theory where $\H_n$ is an ample $Y_n$-module, $\H_Y$ is their direct sum and $Y$ is a coarse disjoint union of $\{Y_n\}$.
\end{lem}

\begin{proof}
Lemma~\ref{lem:cdu to separated case I} and \ref{lem:cdu to separated case II} imply that
	\[
	\frac{C^*(\H_Y)}{\K(\H_Y)} \cong \frac{C^*(\H_Y) \cap \prod_n \B(\H_n)}{\bigoplus_n C^*(\H_n)}
	\quad \mathrm{and} \quad
	\frac{\Csq(\H_Y)}{\K(\H_Y)} \cong \frac{\Csq(\H_Y) \cap \prod_n \B(\H_n)}{\bigoplus_n \Csq(\H_n)}.
	\]
Since $C^*(\H_n)=\Csq(\H_n)$ for each $n$, we obtain the following commutative diagram:\\
	\begin{small}
		\centerline{
			\xymatrix{
				\cdots \ar[r] & K_*\big(\bigoplus_n C^*(\H_n)\big) \ar[r]  \ar@{=}[d] & K_*(C^*(\H_Y) \cap \prod_n \B(\H_n)) \ar[d]\ar[r] & K_*(C^*(\H_Y)/\K(\H_Y)) \ar[d]\ar[r]  & \cdots \\
				\cdots \ar[r] & K_*\big(\bigoplus_n \Csq(\H_n)\big) \ar[r]  & K_*(\Csq(\H_Y) \cap \prod_n \B(\H_n)) \ar[r] & K_*(\Csq(\H_Y)/\K(\H_Y)) \ar[r] & \cdots.
		}}
	\end{small}	
Hence the right vertical map is an isomorphism from the assumption and the Five Lemma. Now consider the following commutative diagram:\\
	\centerline{
		\xymatrix{
			\cdots \ar[r] & K_*(\K(\H_Y)) \ar[r]  \ar@{=}[d] & K_*(C^*(\H_Y)) \ar[d]\ar[r] & K_*(C^*(\H_Y) / \K(\H_Y))  \ar[d]\ar[r]  & \cdots \\
			\cdots \ar[r] & K_*(\K(\H_Y)) \ar[r]  & K_*(\Csq(\H_Y)) \ar[r] & K_*(\Csq(\H_Y) / \K(\H_Y)) \ar[r] & \cdots,
	}}
	we obtain that $K_*(C^*(\H_Y)) \to K_*(\Csq(\H_Y) )$ is an isomorphism by the Five Lemma.
	
	Now for a metric space $X$ satisfying the assumption, we follow the argument in \cite[Lemma 12.5.3]{willett2020higher}. Fix a basepoint $x_0 \in X$ and for each $n \in \NN\cup\{0\}$, we set
	\[
	X_n:=\{x\in X: n^3-n\le d_X(x,x_0)\le (n+1)^3+(n+1)\}.
	\]
	Let $A:=\bigsqcup_{n: even} X_n$ and $B:=\bigsqcup_{n: odd}X_n$. It is obvious that $(A,B)$ is an $\omega$-exicisive cover of $X$. Applying the coarse Mayer-Vietoris sequences for the associated Roe algebras (\cite{HRY93}) and strongly quasi-local algebras (Proposition \ref{prop:M-V sequence for strong quasi-local algebras}), we obtain the following commutative diagram
	\begin{small}
		\[
		\begin{CD}
		\cdots @>>>K_*(C^*(A\cap B)) @>>> K_*(C^*(A))\oplus K_*(C^*(B)) @>>> K_*(C^*(X))@>>> \cdots\\
		& &   @VVV            @VVV            @VVV \\
		\cdots @>>>K_*(\Csq(A\cap B)) @>>> K_*(\Csq(A))\oplus K_*(\Csq(B)) @>>> K_*(\Csq(X))@>>> \cdots.
		\end{CD}
		\]
	\end{small}
The left and middle vertical maps are isomorphisms according to the previous paragraph, hence we conclude the proof by the Five Lemma again.
\end{proof}

\section{Twisted strongly quasi-local algebras}\label{sec:twisted}
In this section, we recall the Bott-Dirac operators which will be used in the next section to construct index maps. We also recall the notion of twisted Roe algebras from \cite[Section 12.3]{willett2020higher} (originally in \cite[Section 5]{Yu00}) and introduce their strongly quasi-local analogue. 

\subsection{The Bott-Dirac operators on Euclidean spaces}\label{ssec:The Bott-Dirac operators on Euclidean spaces}
Let us start by recalling some elementary properties of the Bott-Dirac operators. Here we only list necessary notions and facts, while guide readers to \cite[Section 12.1]{willett2020higher} for details.

Let $E$ be a real Hilbert space (also called a \emph{Euclidean space}) with even dimension $d \in \NN$. The \emph{Clifford algebra} of $E$, denoted by $\Cliff_\CC(E)$, is the universal unital complex algebra containing $E$ as a real subspace and subject to the multiplicative relations $x \cdot x = \|x\|_E^2$ for all $x\in E$. It is natural to treat $\Cliff_\CC(E)$ as a graded Hilbert space (see for example \cite[Example E.2.12]{willett2020higher}), and in this case we denote it by $\H_E$.

Denote $\L _E^2$ the graded Hilbert space of square integrable functions from $E$ to $\H_E$ where the grading is inherited from $\H_E$, and $\S_E$ the dense subspace consisting of Schwartz class functions from $E$ to $\H_E$. Fix an orthonormal basis $\{e_1,\ldots,e_d\}$ of $E$ and let $x_1,\ldots,x_d:E \to \RR$ be the corresponding coordinates. Recall that the \emph{Bott operator} $C$ and the \emph{Dirac operator} $D$ are unbounded operators on $\L _E^2$ with domain $\S_E$ defined as:
\[
(Cu)(x) = x \cdot u(x), \quad \mbox{and} \quad (Du)(x)=\sum_{i=1}^d \hat{e_i} \cdot \frac{\partial u}{\partial x_i}(x)
\]
for $u \in \S_E$ and $x\in E$, where $\hat{e_i}: \Cliff_\CC(E) \to \Cliff_\CC(E)$ is the operator determined by $\hat{e_i} (w)=(-1)^{\partial w} w \cdot e_i$ for any homogeneous element $w\in \Cliff_\CC(E)$.

\begin{defn}\label{defn:Bott-Dirac}
	The \emph{Bott-Dirac operator} is the unbounded operator $B:=D+C$ on $\L _E^2$ with domain $\S_E$.
\end{defn}

Given $x\in E$, recall that the \emph{left Clifford multiplication operator associated to $x$} is the bounded operator $c_x$ on $\L _E^2$ defined as the left Clifford multiplication by the fixed vector $x$, and the \emph{translation operator associated to $x$} is the unitary operator $V_x$ on $\L _E^2$ defined by $(V_xu)(y):=u(y-x)$. Given $s\in [1,\infty)$, recall that the \emph{shrinking operator associated to $s$} is the unitary operator $S_s$ on $\L _E^2$ defined by $(S_s u)(y):=s^{{-d}/{2}}u(sy)$. 

\begin{defn}\label{defn:Bott-Dirac modified}
	For $s\in [1,\infty)$ and $x\in E$, the \emph{Bott-Dirac operator associated to $(x,s)$} is the unbounded operator $B_{s,x}:=s^{-1}D+C-c_x$ on $\L _E^2$ with domain $\S_E$.
\end{defn}

Note that $B_{1,0}=B$ and $B_{s,x}=s^{{-1}/{2}}~V_xS_{\sqrt{s}}BS^*_{\sqrt{s}}V_x^*$. It is also known that for each $s\in [1,\infty)$ and $x\in E$, the operator $B_{s,x}$ is unbounded, odd, essentially self-adjoint and maps $\S_E$ to itself (see, \emph{e.g.}, \cite[Corollary 12.1.4]{willett2020higher}). 

\begin{defn}\label{defn:F_{s,x}}
	Let $s\in [1,\infty)$, $x\in E$ and $B_{s,x}$ be the Bott-Dirac operator associated to $(x,s)$. Define a bounded operator on $\L_E^2$ by:
	\[
	F_{s,x}:= B_{s,x}(1+B^2_{s,x})^{-1/2}.
	\]
\end{defn}


We list several important properties of the operator $F_{s,x}$. For simplicity, denote $\chi_{x,R}:=\chi_{B(x;R)}$ for $x\in E$ and $R \geq 0$.
\begin{prop}[{\cite[Proposition 12.1.10]{willett2020higher}}]\label{prop:Psi function}
	For each $\varepsilon >0$ there exists an odd function $\Psi :\RR \rightarrow [-1,1]$ with $\Psi (t)\rightarrow 1$ as $t\rightarrow +\infty$, satisfying the following:
	\begin{enumerate}
		\item For all $s\in [1,\infty)$ and $x\in E$, we have $\|F_{s,x}-\Psi(B_{s,x})\|< \varepsilon$.
		\item There exists $R_0>0$ such that for all $s\in [1,\infty)$ and $x\in E$, we have $\ppg(\Psi(B_{s,x})) \le s^{-1}R_0$.
		\item For all $s\in [1,\infty)$ and $x\in E$, $\Psi(B_{s,x})^2 - 1$ is compact.
		\item For all $s\in [1,\infty)$ and $x,y\in E$, $\Psi(B_{s,x}) - \Psi(B_{s,y})$ is compact.
		\item For all $s\in [1,\infty)$ and $x,y\in E$, $\|F_{s,x}-F_{s,y}\|\le 3\|x-y\|_E$. And there exists $c>0$ such that for all $s\in [1,\infty)$ and $x,y\in E$, we have
		\[
		\|\Psi(B_{s,x}) - \Psi(B_{s,y})\| \leq c \|x-y\|_E.
		\]
		\item For all $x\in E$, the function 
		$$[1,\infty)\rightarrow \B(\L_E^2),\ s \mapsto \Psi(B_{s,x})$$
		is strong-$*$ continuous.
		\item The family of functions
		\[
		[1,\infty) \to \B(\L_E^2), \ s \mapsto \Psi(B_{s,x})^2-1
		\]
		is norm equi-continuous as $x$ varies over $E$ and $s$ varies over any fixed compact subset of $[1,\infty)$.
		\item For any $r \geq 0$, the family of functions
		\[
		[1,\infty) \to \B(\L_E^2), \ s \mapsto \Psi(B_{s,x}) - \Psi(B_{s,y})
		\] 
		is norm equi-continuous as $(x,y)$ varies over the elements of $E \times E$ with $|x-y| \leq r$, and $s$ varies over any fixed compact subset of $[1,\infty)$.
		\item For any $\varepsilon_1>0$, there exists $R_1 > 0$ such that for all $R\ge R_1$, $s\ge 2d$ and $x\in E$, we have
		$$\| ( \Psi(B_{s,x})^2-1 )( 1-\chi_{x,R} ) \|< \varepsilon_1.$$
		\item For any $\varepsilon_2>0, r>0$ there exists $R_2 >0$ such that for all $R\ge R_2$, $s\ge 2d$ and $x,y \in E$ with $\|x-y\|_E\le r$, we have
		$$\|(\Psi(B_{s,x})-\Psi(B_{s,y}))(1-\chi_{x,R})\|< \varepsilon_2.$$
	\end{enumerate}
	Moreover, we can require the function $\Psi$, constants $R_0$ in (2), $c$ in (5), $R_1$ in (9) and $R_2$ in (10) are independent of the dimension $d$ of the Euclidean space $E$.
\end{prop}

\begin{rem}\label{rem:Psi function revised}
Note that statements (9) and (10) above are slightly stronger than those in \cite[Proposition 12.1.10]{willett2020higher}. For completeness, we give the proofs in Appendix \ref{app:proof2}.
\end{rem}

\subsection{Twisted Roe and strongly quasi-local algebras}\label{ssec:Twisted Roe and strong quasi-local algebras}
Thanks to Lemma \ref{lem:reduce to block-diagonal}, we only focus on sequences of finite metric spaces with uniformly bounded geometry.

We fix some notation first. 
Let $\{X_n\}_{n\in \NN}$ be a sequence of finite metric spaces with uniformly bounded geometry which admits a uniformly coarse embedding into Euclidean spaces $\{f_n:X_n\rightarrow E_n\}$ where each $E_n$ is a Euclidean space of even dimension $d_n$. 
Let $X$ be a coarse disjoint union of $\{X_n\}$ and denote $E:=\{E_n\}_{n\in \NN}$.

Recall that $\HH$ is a fixed infinite-dimensional separable Hilbert space. Denote $\H_n:=\ell^2(X_n)\otimes \HH$, which is an ample $X_n$-module under the amplified multiplication representation. Denote $\H_{n,E_n}:=\H_n\otimes \L^2_{E_n}$, which is both an ample $X_n$-module and an ample $E_n$-module similarly. Also define $\H_X:=\bigoplus_n\H_n$ and $\H_{X,E}:=\bigoplus_n\H_{n,E_n}$, both of which are ample $X$-modules. For $T_n\in \B(\H_{n,E_n})$, write $\ppg_{X_n}(T_n)$ and $\ppg_{E_n}(T_n)$ for the propagation of $T_n$ with respect to the $X_n$-module structure and the $E_n$-module structure, respectively. From Definition \ref{defn:Roe alg} and Definition \ref{defn:strongly quasi-local algebra}, we form the Roe algebras $C^*(\H_{n,E_n})$ of $X_n$ and $C^*(\H_{X,E})$ of $X$, and the strongly quasi-local algebras $\Csq(\H_{n,E_n})$ of $X_n$ and $\Csq(\H_{X,E})$ of $X$.

To introduce the twisted Roe and strongly quasi-local algebras, we need an extra construction from \cite[Definition 12.3.1]{willett2020higher} which involves the information of uniformly coarse embedding as follows:

\begin{defn}\label{defn:V-construction}
	Given $n\in\NN$ and $T \in \B(\L_{E_n}^2)$, we define a bounded operator $T^V$ on $\H_{n,E_n}=\ell^2(X_n)\otimes \HH\otimes \L^2_{E_n}$ by the formula
	\[
	T^V: \delta_x\otimes \xi\otimes u\mapsto \delta_x\otimes \xi\otimes V_{f_n(x)}TV_{f_n(x)}^*u,
	\]
	for $x\in X_n$, $\xi\in \HH$ and $u\in \L^2_{E_n}$, where $f_n$ is the uniformly coarse embedding and $V_{f_n(x)}$ is the translation operator defined in Section \ref{ssec:The Bott-Dirac operators on Euclidean spaces}.
\end{defn}

For each $n$, decompose $\H_n=\bigoplus_{x\in X_n} \H_{n,x}$ where $\H_{n,x} := \chi_{\{x\}} \H_n$ for $x\in X_n$ and $\H_{n,E_n}=\bigoplus_{x\in X_n} \H_{n,x} \otimes \L_{E_n}^2$. Hence $T\in \B(\H_{n,E_n})$ can be considered as an $X_n$-by-$X_n$ matrix operator $(T_{x,y})_{x,y\in X_n}$ where $T_{x,y}$ is a bounded operator from $\H_{n,y} \otimes \L_{E_n}^2$ to $\H_{n,x} \otimes \L_{E_n}^2$. It is clear that for $T \in \B(\L_{E_n}^2)$ we have:
\[
T^V_{x,y}=
\begin{cases}
~\Id_{\HH} \otimes V_{f_n(x)}TV_{f_n(x)}^*, & y=x; \\
~0, & \mbox{otherwise}.
\end{cases}
\]
Hence $T^V$ is a block-diagonal operator with respect to the above decomposition. 

Now we introduce the following twisted Roe algebras from \cite[Section 12.6]{willett2020higher}. 

\begin{defn}\label{defn:twisted Roe}
	Let $\prod_{n\in \NN}C_b([1,\infty),\K(\H_{n,E_n}))$ denote the product $C^*$-algebra of all bounded continuous functions from $[1,\infty)$ to $\K(\H_{n,E_n})$ with supremum norm. Write elements of this $C^*$-algebra as a collection $(T_{n,s})_{n\in \NN,s\in [1,\infty)}$ for $T_{n,s} \in \K(\H_{n,E_n})$, whose norm is
	\[
	\|(T_{n,s})\|=\sup_{n\in \NN,s\in [1,\infty)}\|T_{n,s}\|_{\B(\H_{n,E_n})}.
	\]
	Let $\AA(X;E)$ denote the $*$-algebra of $\prod_{n\in \NN}C_b([1,\infty),\K(\H_{n,E_n}))$ consisting of elements satisfying the following conditions:
	\begin{enumerate}
		\item $\sup\limits_{s\in [1,\infty),n\in \NN}\ppg_{X_n}(T_{n,s})<\infty$;\\[0.1cm]
		\item $\lim\limits_{s\rightarrow \infty}\sup\limits_{n\in \NN}\ppg_{E_n}(T_{n,s})=0$;\\[0.1cm]
		\item $\lim\limits_{R\rightarrow \infty}\sup\limits_{s\in [1,\infty),n\in \NN}\|\chi_{0,R}^VT_{n,s}-T_{n,s}\|=\lim\limits_{R\rightarrow \infty}\sup\limits_{s\in [1,\infty),n\in \NN}\|T_{n,s}\chi_{0,R}^V-T_{n,s}\|=0$.
	\end{enumerate}
	The \emph{twisted Roe algebra} $A(X;E)$ of $\{X_n\}_{n\in \NN}$ is defined to be the norm-closure of $\AA(X;E)$ in $\prod_{n\in \NN}C_b([1,\infty),\K(\H_{n,E_n}))$.
\end{defn}

\begin{rem}\label{rem:diff between twisted Roe}
	The above definition appears different from \cite[Definition 12.6.2]{willett2020higher}, while it coincides with the case of $r=0$ therein. To see this, note that each $X_n$ is finite hence $C^*(\H_{n,E_n})=\K(\H_{n,E_n})$ for $n\in \NN$. Then the following lemma shows that we can recover condition (4) in \cite[Definition 12.6.2]{willett2020higher}.
\end{rem}

\begin{lem}\label{lem:condition (4) recovery}
	Given $n\in \NN$ and a compact operator $T \in \K(\H_{n,E_n})$, we have
	\[
	\lim_{i\in I}\|p_i^VT-T\|=\lim_{i\in I}\|Tp_i^V-T\|=0
	\]
	where $\{p_i\}_{i\in I}$ is the net of finite rank projections on $\L_{E_n}^2$. 
\end{lem}

\begin{proof}
	Given $\varepsilon>0$, it suffices to find a finite rank projection $p\in \B(\L_{E_n}^2)$ such that $\| p^V T-T \| < \varepsilon$. Replacing $T$ by its adjoint $T^*$, we obtain the other equality as well. 
	
	Since $T$ is compact, there exists a finite rank projection $P\in \B(\H_{n,E_n})$ such that $\|PT-T\| < \frac{\varepsilon}{2}$. Moreover, we can assume that the image of $P$ is contained in the subspace spanned by the finite set:
	\[
	\big\{ \delta_{x} \otimes \xi_i \otimes u_j: x \in X_n, \xi_i\in \HH, u_j \in \L_E^2 \mbox{~for~}i, j=1,2,\ldots,N\big\}.
	\]
	Hence for each $x \in X_n$, there exists a finite rank projection $q_x\in \B(\L_{E_n}^2)$ such that 
	\[
	P\le \sum_{x\in X_n} p_x \otimes \Id_{\HH} \otimes q_x,
	\]
	where $p_x$ is the orthogonal projection onto $\CC \delta_{x} \subseteq \ell^2(X_n)$. Take an arbitrary finite rank projection $p \in \B(\L^2_E)$ with $p\geq V_{f_n(x)}^* q_x V_{f_n(x)}$ for each $x\in X_n$. Then we have:
	\begin{align*}
	p^V &= \sum_{x\in X_n} p_x \otimes \Id_{\HH} \otimes V_{f_n(x)} p V_{f_n(x)}^* 
	\ge \sum_{x\in X_n} p_x \otimes \Id_{\HH} \otimes q_x \ge P
	\end{align*}
	This implies that $p^VP= P$. Hence we obtain
	\[
	\|p^VT - T\| \le \|p^VT - p^VPT\| + \|PT - T\| \le 2\|PT - T\| < \varepsilon,
	\]
	which concludes the proof.
\end{proof}


\begin{defn}\label{defn:twisted quasi-local}
	Let $\prod_{n\in \NN}C_b([1,\infty),\K(\H_{n,E_n}))$ denote the product $C^*$-algebra of all bounded continuous functions from $[1,\infty)$ to $\K(\H_{n,E_n})$ with supremum norm. Write elements of this $C^*$-algebra as a collection $(T_{n,s})_{n\in \NN,s\in [1,\infty)}$ for $T_{n,s} \in \K(\H_{n,E_n})$, whose norm is
	\[
	\|(T_{n,s})\|=\sup_{n\in \NN,s\in [1,\infty)}\|T_{n,s}\|_{\B(\H_{n,E_n})}.
	\]
	Let $\AAsq (X;E)$ denote the $*$-algebra of $\prod_{n\in \NN}C_b([1,\infty),\K(\H_{n,E_n}))$ consisting of elements satisfying the following conditions:
	\begin{enumerate}
		\item For any $\varepsilon>0$, there exists $\delta, R >0$ such that for any $n\in \NN, s \in [1,\infty)$ and $g_n:X_n\rightarrow \K(\HH)_1$ with $(\delta, R)$-variation we have $\|[T_{n,s}\otimes\Id_{\HH},\Lambda_{\H_{n,E_n}}(g_n)]\|<\varepsilon$, where $\HH$ is the fixed Hilbert space and $\Lambda$ is from (\ref{EQ:Lambda}). \\[0.05cm]
		\item $\lim\limits_{s\rightarrow \infty}\sup\limits_{n\in \NN}\ppg_{E_n}(T_{n,s})=0$.\\[0.05cm]
		\item For any $\varepsilon>0$, there exists $R'>0$ such that for each $n$, $C_n\subseteq X_n$ and Borel set $D_n\subseteq E_n$ with $d_{E_n}(f_n(C_n),D_n)\ge R'$ we have $\|\chi_{C_n}T_{n,s}\chi_{D_n}\|<\varepsilon$ and $\|\chi_{D_n}T_{n,s}\chi_{C_n}\|<\varepsilon$ for all $s \in [1,\infty)$.
	\end{enumerate}
	The \emph{twisted strongly quasi-local algebra} $\Asq(X;E)$ of $\{X_n\}_{n\in \NN}$ is defined to be the norm-closure of $\AAsq (X;E)$ in $\prod_{n\in \NN}C_b([1,\infty),\K(\H_{n,E_n}))$.
\end{defn}

\begin{rem}\label{rem:Asq explanation}
	We provide some explanation on condition (3) in Definition \ref{defn:twisted quasi-local}. Recall that $\H_{n,E_n}$ is both an $X_n$-module and an $E_n$-module, so we can consider the \emph{$(X_n \times E_n)$-support} of a given operator $T \in \B(\H_{n,E_n})$ defined as 
	\[
	\supp_{X_n \times E_n}(T):=\big\{(x,v) \in X_n \times E_n: \chi_{\{x\}} T\chi_U \neq 0 \mbox{~for~all~neighbourhoods~}U\mbox{~of~}v\big\}.
	\]
	We define the associated \emph{$(X_n \times E_n)$-propagation} of $T$ to be
	\[
	\ppg_{X_n,E_n}(T):=\sup\big\{\|f_n(x)-v\|_{E_n}: (x,v) \in \supp_{X_n \times E_n}(T)\big\}.
	\]
Definition \ref{defn:twisted quasi-local}(3) says that $T_{n,s}$ and $T_{n,s}^*$ are \emph{uniformly $(X_n \times E_n)$-quasi-local} in the sense that for any $\varepsilon>0$ there exists $R>0$ such that for each $n$, $C_n\subseteq X_n$ and Borel set $D_n \subseteq E_n$ with $C_n \times D_n \subseteq \big\{(x,v) \in X_n \times E_n: \|f_n(x)-v\|_E \geq R\big\}$, we have $\|\chi_{C_n}T_{n,s}\chi_{D_n}\|<\varepsilon$ and $\|\chi_{C_n}T^*_{n,s}\chi_{D_n}\|<\varepsilon$ for all $s \in [1,\infty)$. It is clear that limits of uniformly finite $(X_n \times E_n)$-propagation operators are uniformly $(X_n \times E_n)$-quasi-local. 
\end{rem}

\begin{lem}\label{lem:A and Aq}
	We have $A(X;E) \subseteq \Asq (X;E) \subseteq \prod_{n\in \NN}C_b([1,\infty),\K(\H_{n,E_n}))$.
\end{lem}

\begin{proof}
	Given $T=(T_{n,s})_{n\in \NN,s\in [1,\infty)} \in \AA(X;E)$, condition (1) in Definition \ref{defn:twisted quasi-local} follows from Proposition~\ref{prop:relations between Roe, QL and SQL} and Lemma \ref{lem:char for sequence of strong quasi-locality}. We only need to check condition (3). Given $R>0$, Remark \ref{rem:Asq explanation} implies that it suffices to show that $T_{n,s}\chi_{0,R}^V$ and $(T_{n,s}\chi_{0,R}^V)^\ast= \chi_{0,R}^VT_{n,s}^\ast$ have uniformly finite $(X_n \times E_n)$-propagation for $n\in \NN$ and $s\in [1,\infty)$. Now Definition \ref{defn:twisted Roe}(1) implies that there exists an $M>0$ such that $\ppg_{X_n}(T_{n,s}) \leq M$ and $\ppg_{X_n}(T_{n,s}^\ast) \leq M$ for all $n\in \NN$ and $s\in [1,\infty)$. Since $\{f_n: X_n \to E_n\}_{n\in \NN}$ is a uniformly coarse embedding, there exists some $\rho_+: \RR^+ \to \RR^+$ such that $\|f_n(x)-f_n(y)\|_E \leq \rho_+(d_{X_n}(x,y))$ for $n\in \NN$ and $x,y\in X_n$. It follows directly from definition that $\ppg_{X_n,E_n}(T_{n,s}\chi_{0,R}^V) \leq \rho_+(M)+R$ and $\ppg_{X_n,E_n}(\chi_{0,R}^VT_{n,s}^\ast) \leq \rho_+(M)+R$ for all $n\in \NN$ and $s\in [1,\infty)$.
\end{proof}

Finally, we introduce the following operators:

\begin{defn}[{\cite[Section 12.3 and 12.6]{willett2020higher}}]\label{defn:F}
	For each $n\in \NN$, $s\in [1,\infty)$ and $x\in E_n$, Definition \ref{defn:F_{s,x}} provides a bounded operator $F_{s,x} \in \B(\L_{E_n}^2)$, also denoted by $F_{n,s,x}$. Applying Definition \ref{defn:V-construction}, we obtain an operator $F_{n,s}:=F^V_{n,s+2d_n,0}$ in $\B(\H_{n,E_n})$ where $d_n$ is the dimension of $E_n$. Let $F_s$ be the block diagonal operator in $\prod_n \B(\H_{n,E_n}) \subseteq \B(\H_{X,E})$ defined by $F_s:=(F^V_{n,s+2d_n,0})_n$. Finally, we define $F$ to be an element in $\prod_n \B(L^2([1,\infty);\H_{n,E_n})) \subseteq \B(L^2 ([1,\infty);\H_{X,E}))$ defined by $(F(u))(s):=F_su(s)$.
	
	Similarly, given $\varepsilon>0$ let $\Psi$ be a function as in Proposition~\ref{prop:Psi function} and $F^\Psi_s$ be the bounded diagonal operator on $\H_{X,E}$ defined by $F^\Psi_s:=(F^\Psi_{n,s})_n$ where $F^\Psi_{n,s}=\Psi(B_{n,s+2d_n,0})^V$. Let $F^\Psi$ be the bounded operator on $\bigoplus_n L^2 ([1,\infty), \H_{n,E_n} )$ defined by $(F^\Psi(u))(s):=F^\Psi_su(s)$.
\end{defn}

\section{The index map}\label{sec:index map}

Recall that in \cite[Secition 12.3 and 12.6]{willett2020higher}, Willett and Yu construct an index map (with notation as in Section \ref{ssec:Twisted Roe and strong quasi-local algebras}):
\[
\Ind_F: K_* \big( C^*(\H_X)\cap \prod_n\B(\H_n) \big) \rightarrow K_*(A(X;E)),
\]
where $F$ is the operator in Definition \ref{defn:F}. They use $\Ind_F$ to transfer $K$-theoretic information from Roe algebras to their twisted counterparts, which allow them to reprove the coarse Baum-Connes conjecture via local isomorphisms. More precisely, they prove the following:

\begin{prop}[{\cite[Proposition 12.6.3]{willett2020higher}}]\label{prop:index map isom. for Roe}
	With notation as in Section \ref{ssec:Twisted Roe and strong quasi-local algebras}, for each $s\in [1,\infty)$ the composition
	\[
	K_* \big( C^*(\H_X)\cap \prod_n\B(\H_n) \big) \stackrel{\Ind_F}{\longrightarrow} K_*(A(X;E)) \stackrel{\iota^s_*}{\longrightarrow} K_*\big( C^*(\H_{X,E})\cap \prod_n\B(\H_{n,E_n}) \big)
	\]
	is an isomorphism, where $\iota^s:A(X;E)\rightarrow C^*(\H_{Y,E})$ is the evaluation map at $s$.
\end{prop}

In this section, we construct index maps in the strongly quasi-local setting and prove similar results. This allows us to prove certain isomorphisms in $K$-theory to attack Theorem \ref{thm:main result} later. We follow the procedure in \cite[Section 12.3]{willett2020higher}, while more technical analysis is required.


We follow the same notation as in Section \ref{ssec:Twisted Roe and strong quasi-local algebras}. Let $\{X_n\}_{n\in \NN}$ be a sequence of finite metric spaces with uniformly bounded geometry which admits a uniformly coarse embedding into Euclidean spaces $\{f_n:X_n\rightarrow E_n\}$ where each $E_n$ is a Euclidean space of even dimension $d_n$. 

Let us start with several lemmas to analyse relations between the operator $F$ from Definition \ref{defn:F} and the twisted strongly quasi-local algebra $\Asq(X;E)$. 

\begin{lem}\label{lem:F multiplier}
	The operator $F$ is a self-adjoint, norm one, odd operator in the multiplier algebra of $\Asq(X;E)$. 
\end{lem}

\begin{proof}
	The operator $F$ is self-adjoint, norm one and odd since each $F_{n,s,x}$ is. Given $\varepsilon>0$, let $\Psi:\RR \to [-1,1]$ be a function as in Proposition \ref{prop:Psi function} for this $\varepsilon$. Then Proposition \ref{prop:Psi function}(1) implies:
	\[
	\|F - F^\Psi\| \leq \sup_{n\in \NN, s\in [1,\infty)} \|F^V_{n,s,0} - \Psi(B_{n,s,0})^V\| \leq \sup_{n\in \NN, s\in [1,\infty)} \sup_{x\in X_n} \|F_{n,s,f_n(x)} - \Psi(B_{n,s,f_n(x)})\| \leq \varepsilon.
	\]
Hence it suffices to show that $(T_{n,s}) F^\Psi=(T_{n,s}F^\Psi_{n,s})$ belongs to $\AAsq(X;E)$ for any $(T_{n,s}) \in \AAsq(X;E)$.
	
First it follows from \cite[Lemma 12.3.5]{willett2020higher} that for each $n\in \NN$, the map $s\mapsto T_{n,s}F^\Psi_{n,s}$ is norm-continuous. 
Now we verify conditions (1)-(3) in Definition \ref{defn:twisted quasi-local} for $(T_{n,s}F^\Psi_{n,s})$. Note that condition (2) follows directly from Proposition \ref{prop:Psi function}(2) and (3) holds since $\ppg_{E_n}(F^\Psi_{n,s})$ are uniformly finite for all $n\in \NN$ and $s \in [1,\infty)$. For condition (1), note that for any $n\in\NN$, $s\in [1,\infty)$ and $g:X_n\rightarrow\K(\HH)_1$, we have
\[
\big( F^\Psi_{n,s}\otimes\Id_{\HH} \big) \cdot \Lambda(g) = \Lambda(g) \cdot \big( F^\Psi_{n,s}\otimes\Id_{\HH} \big).
\]
Hence we obtain:
\[
		\big\|[(T_{n,s}F^\Psi_{n,s})\otimes\Id_{\HH}, \Lambda(g)]\big\|=\big\|\big([T_{n,s}\otimes\Id_{\HH},\Lambda(g)]\big)\cdot \big(F^\Psi_{n,s}\otimes\Id_{\HH}\big)\big\|  \le \big\|[T_{n,s}\otimes\Id_{\HH},\Lambda(g)]\big\|,
\]
which concludes the proof.
\end{proof}

\begin{lem}\label{lem:multiplier}
	Considered as represented on $L^2([1,\infty)) \otimes\H_{X,E}$ via amplification of identity, $\Csq(\H_X)\cap \prod_n\B(\H_n)$ is a subalgebra of the multiplier algebra of $\Asq(X;E)$.
\end{lem}

\begin{proof}
	It suffices to show that $(S_nT_{n,s}) \in \AAsq(X;E)$ for any $(T_{n,s}) \in \AAsq(X;E)$ and $(S_n)\in \Csq(\H_X)\cap \prod_n\B(\H_n)$.\footnote{To be more precise, $(S_nT_{n,s})$ stands for $\big((S_n\otimes \Id_{\L^2_{E_n}}) \cdot T_{n,s}\big)$.} It is clear that the map $s\mapsto S_nT_{n,s}$ is norm-continuous and bounded for each $n\in \NN$. 

Now we verify conditions (1)-(3) in Definition \ref{defn:twisted quasi-local} for $s\mapsto (S_nT_{n,s})$. First note that for any $n\in\NN$, $s\in [1,\infty)$ and $g:X_n\rightarrow\K(\HH)_1$ we have
\begin{equation}\label{EQ:commutant lemma}
	\big\| [S_n\otimes \Id_{\L^2_{E_n}}\otimes \Id_{\HH}, \Lambda_{\H_{n,E_n}}(g)] \big\| = \big\| [S_n\otimes \Id_{\HH}, \Lambda_{\H_n}(g)] \big\|.
\end{equation}
Hence condition (1) follows from direct calculations together with Lemma \ref{lem:char for sequence of strong quasi-locality}.
%
%

Condition (2) follows from the fact that each $S_n\otimes \Id_{\L^2_{E_n}}$ has zero $E_n$-propagation. Now we check condition (3). Given $\varepsilon>0$, it follows from $(S_n)\in \Csq(\H_X) \subseteq C^*_q(\H_X)$ that there exists $R_1>0$ such that $S_n$ has $(\frac{\varepsilon}{2\|(T_{n,s})\|},R_1)$-propagation for all $n\in\NN$. On the other hand, there exists $R_2>0$ such that for any  $n\in\NN$, $s\in [1,\infty)$, $C_n\subseteq X_n$ and Borel set $D_n\subseteq E_n$ with $d_{E_n}(f_n(C_n),D_n)\ge R_2$ we have $\|\chi_{C_n}T_{n,s}\chi_{D_n}\|<\frac{\varepsilon}{2\|(S_n)\|}$ and $\|\chi_{D_n}T_{n,s}\chi_{C_n}\|<\frac{\varepsilon}{2\|(S_n)\|}$. Now let $R=\rho_+ (R_1) + R_2$	where $\rho_+$ comes from the uniformly coarse embedding $\{ f_n : X_n \rightarrow E_n \}$. For any $n\in\NN$, $C_n'\subseteq X_n$ and Borel set $D_n'\subseteq E_n$ with $d_{E_n}(f_n(C_n'),D_n')\ge R$ we have $f_n( \Nd_{R_1}(C_n') ) \subseteq \Nd_{\rho_+(R_1)}( f_n(C_n') )$, which implies that $d_{E_n}(f_n( \Nd_{R_1}(C_n') ) , D_n')\ge R_2$. Therefore, we obtain:
	\begin{align*}
	\| \chi_{C_n'} S_nT_{n,s}\chi_{D_n'} \| 
	\le &  \|\chi_{C_n'} S_n T_{n,s}\chi_{D_n'} - \chi_{C_n'} S_n \chi_{\Nd_{R_1}(C_n')} T_{n,s}\chi_{D_n'} \| + \| \chi_{C_n'} S_n \chi_{\Nd_{R_1}(C_n')} T_{n,s}\chi_{D_n'} \| \\
	\le & \| \chi_{C_n'} S_n \chi_{(\Nd_{R_1}(C_n'))^c} \| \cdot \|T_{n,s}\| + \| S_n \| \cdot \| \chi_{\Nd_{R_1}(C_n')} T_{n,s}\chi_{D_n'} \| \\
	< & \frac{\varepsilon}{2\|(T_{n,s})\|}  \cdot \| (T_{n,s}) \| + \|(S_n)\| \cdot \frac{\varepsilon}{2\|(S_n)\|} = \varepsilon
	\end{align*}
	for all $s\in [1,\infty)$. On the other hand, we have:
	\[
	\| \chi_{D_n'} S_n T_{n,s} \chi_{C_n'} \| = \| S_n \chi_{D_n'} T_{n,s} \chi_{C_n'} \| \le \|S_n\| \cdot \| \chi_{D_n'} T_{n,s} \chi_{C_n'} \| < \varepsilon
	\]
	for all $s\in [1,\infty)$. So we finish the proof.
\end{proof}

Regarding $\Csq(\H_X)\cap \prod_n\B(\H_n)$ as a subalgebra in $\B(L^2([1,\infty)) \otimes\H_{X,E})$ as in Lemma \ref{lem:multiplier}, we have the following:

\begin{lem}\label{lem:commutator}
For any $(S_n) \in \Csq(\H_X)\cap \prod_n\B(\H_n)$, we have $[(S_n),F] \in \Asq(X;E)$.
\end{lem}

\begin{proof}
From Proposition~\ref{prop:Psi function}(1), it suffices to show that the map
\[
s\mapsto [(S_n),F_s^\Psi] = [(S_n) , (\Psi(B_{n,s+2d_n,0})^V) ]
\]
belongs to $\AAsq(X;E)$ for any $\Psi$ as in Proposition~\ref{prop:Psi function}, \emph{i.e.}, to verify conditions (1)-(3) in Definition \ref{defn:twisted quasi-local}.

First note that for any $n\in\NN$, $s\in[1,\infty)$ and $g:X_n\rightarrow \K(\HH)_1$ we have
\begin{align*}
	&\big[ [S_n,\Psi(B_{n,s+2d_n,0})^V]\otimes\Id_{\HH}, \Lambda_{\H_{n,E_n}}(g) \big] \\
	 &=[S_n\otimes \Id_{\L^2_{E_n}}\otimes\Id_{\HH}, \Lambda_{\H_{n,E_n}}(g)]\Psi(B_{n,s+2d_n,0})^V+ \Psi(B_{n,s+2d_n,0})^V[\Lambda_{\H_{n,E_n}}(g), S_n\otimes \Id_{\L^2_{E_n}}\otimes\Id_{\HH}],
\end{align*}
which has norm at most $2\|[S_n\otimes \Id_{\HH}, \Lambda_{\H_{n}}(g)]\|$ according to (\ref{EQ:commutant lemma}). Hence we conclude condition (1) from the strong quasi-locality of $(S_n)$. Condition (2) follows from Propostion~\ref{prop:Psi function}(2) and that fact that $S_n$ has zero $E_n$-propagation.

To check condition (3), we fix an $\varepsilon>0$. It follows from Proposition~\ref{prop:subspace strong quasi-locality} that there exist $\delta', R'>0$ such that for any $n\in\NN$, $A\subseteq X_n$ and $g:A\rightarrow \K(\HH)_1$ with $(\delta', R')$-variation we have $\big\|[\chi_{A}S_n\chi_{A}\otimes\Id_{\HH},\Lambda(g)]\big\|<\frac{\varepsilon}{4}$. Moreover since $\Csq(\H_X)\subseteq C^*_q(\H_X)$, we assume that $(S_n)$ has $(\frac{\varepsilon}{4},R')$-propagation. Denote by $\rho_+$ the parameter function from the uniformly coarse embedding $\{ f_n : X_n \rightarrow E_n \}$.

By Proposition~\ref{prop:Psi function}(10), there exists $R''>0$ such that for all $n\in\NN, s\geq 2d_n$ and $ x,y\in E_n$ with $\|x-y\|_{E_n} \le\rho_+(R')$ we have $\|(\Psi(B_{n,s,x}) -\Psi(B_{n,s,y}))(1-\chi_{B(x,R'')})\|<\delta'$. Set $R=R''+\rho_+(R')$. 
For any $n\in\NN$, $s\in [1,\infty)$, $C\subseteq X_n$ and Borel set $D\subseteq E_n$ with $d_{E_n}(f_n(C),D)\ge R$, we are going to estimate the norm $\|\chi_C[S_n,F_{n,s}^\Psi]\chi_D\|$. 

Denote $C':=\Nd_{R'}(C) \subseteq X_n$. Since $(S_n)$ has $(\frac{\varepsilon}{4},R')$-propagation, we obtain:
\begin{equation}\label{EQ:commutant lem2}
	\big\|\chi_C[S_n,F_s^\Psi]\chi_D\big\| < 2\cdot \frac{\varepsilon}{4}+ \big\|\chi_C[\chi_{C'}S_n\chi_{C'},F_s^\Psi \chi_D]\big\|.
\end{equation}
Consider the function 
\[
g: X_n \to \B(\L_{E_n}^2)_1 \quad \mbox{by} \quad x \mapsto \Psi(B_{n,s+2d_n,f_n(x)}) \chi_D.
\]
Proposition~\ref{prop:Psi function}(4) implies that $g(x)-g(y) \in \K(\L_{E_n}^2)$ for any $x,y\in X_n$. Moreover, we claim that $g$ has $(\delta',R')$-variation on $C'$. In fact, for any $x,y\in C'$ with $d_{X_n}(x,y)< R'$ we have $\|f_n(x)-f_n(y)\|_{E_n}\le \rho_+(R')$. Note that $d_{E_n}(f_n(C),D)\ge R$ and $x\in C'=\Nd_{R'}(C)$, hence $D \subseteq E_n \setminus B(f_n(x),R'')$. Therefore by the choice of $R''$ above, we obtain that $g$ has $(\delta',R')$-variation on $C'$. 

Finally, note that each $\L_{E_n}^2$ is separable and infinite dimensional, hence isomorphic to the fixed Hilbert space $\HH$. Fixing an $x_0 \in X_n$, we define $\hat{g}: X_n \to \K(\L_{E_n}^2)_1$ by $\hat{g}(x):=\frac{g(x)-g(x_0)}{2}$. It follows from the above analysis that $\hat{g}$ has $(\delta',R')$-variation on $C'$. Hence by the choice of $\delta', R'$ at the beginning, we obtain that
\[
[\chi_{C'}S_n\chi_{C'},F_s^\Psi \chi_D] = \big[(\chi_{C'}S_n \chi_{C'}) \otimes \Id_{\L_{E_n}^2}, 2\Lambda_{\H_n}(\hat{g})\big]
\]
has norm at most $\frac{\varepsilon}{2}$. Combining with (\ref{EQ:commutant lem2}), we obtain:
\[
\big\|\chi_C[S_n,F_s^\Psi]\chi_D\big\| < 2\cdot \frac{\varepsilon}{4}+ \big\|\chi_C[\chi_{C'}S_n\chi_{C'},F_s^\Psi \chi_D]\big\| \leq \frac{\varepsilon}{2}+\frac{\varepsilon}{2} = \varepsilon.
\]
Similarly, we have $\|\chi_D[S_n,F_{n,s}^\Psi]\chi_C\| < \varepsilon$. Hence we conclude the proof.
\end{proof}

\begin{lem}\label{lem:projection}
	For any projection $(p_n) \in \Csq(\H_X)\cap \prod_n\B(\H_n)$, the function
	\[
	s\mapsto ( (p_n)F_s(p_n) )^2 - (p_n)
	\]
	is in $(p_n) A_{sq}(X;E) (p_n)$.
\end{lem}

\begin{proof}
From Lemma~\ref{lem:commutator}, it suffices to show that the function $ s\mapsto  (p_n)F_s^2 - (p_n)$ is in $\Asq(X;E)$. Moreover, we only need to show that the function
\[
s\mapsto (p_n) (F_{n,s}^\Psi)^2 - (p_n) = (p_n (\Psi(B_{n,s+2d_n,0})^V)^2) - (p_n)
\]
is in $\AAsq(X;E)$ for any $\Psi$ as in Proposition~\ref{prop:Psi function}. For $n\in \NN$, it follows from Proposition \ref{prop:Psi function}(7) that the function $s \mapsto p_n(F_{n,s}^\Psi)^2 - p_n$ is bounded and continuous.

Now we verify conditions (1)-(3) in Definition \ref{defn:twisted quasi-local}.
First note that for any $n\in\NN$, $s\in[1,\infty)$ and $g:X_n\rightarrow \K(\HH)_1$ we have
\begin{align*}
	\big[&(p_n(\Psi(B_{n,s+2d_n,0})^V)^2-p_n)\otimes\Id_{\HH}, \Lambda_{\H_{n,E_n}}(g)\big] \\
	&= \big[p_n \otimes\Id_{\L^2_{E_n}}\otimes\Id_{\HH}, \Lambda_{\H_{n,E_n}}(g)\big]\cdot \big((\Psi(B_{n,s+2d_n,0})^V)^2\otimes\Id_{\HH}\big) - \big[p_n \otimes\Id_{\L^2_{E_n}} \otimes\Id_{\HH}, \Lambda_{\H_{n,E_n}}(g)\big],
\end{align*}
which has norm at most $2\|[p_n\otimes \Id_{\HH}, \Lambda_{\H_{n}}(g)]\|$ according to (\ref{EQ:commutant lemma}). Hence we conclude condition (1) from the strong quasi-locality of $(p_n)$. Condition (2) follows from Propostion~\ref{prop:Psi function}(2) and that fact that $p_n$ has zero $E_n$-propagation. Finally, condition (3) follows from the uniform quasi-locality of $(p_n)$ together with Proposition~\ref{prop:Psi function}(9). Hence we conclude the proof.
\end{proof}

Having obtained the above essential ingredients, we are now in the position to construct the index map. It follows from a standard construction in $K$-theories (see, \emph{e.g.}, \cite[Definitoin 2.8.5]{willett2020higher}):

\begin{defn}\label{defn:index definition}
	Let $\H=\H^{+}\oplus \H^{-}$ be a graded Hilbert space with grading operator $U$ (\emph{i.e.}, $U$ is a self-adjoint unitary operator in $\B(\H)$ such that $\H^{\pm}$ coincides with the $(\pm 1)$-eigenspace of $U$), and $A$ be a $C^*$-subalgebra of $\B(\H)$ such that $U$ is in the multiplier algebra of $A$. Let $F\in \B(\H)$ be an odd operator of the form
	\[ F = 
	\begin{pmatrix}
	0&V \\
	W&0
	\end{pmatrix}
	\]
	for some operators $V: \H^{-} \rightarrow \H^{+}$ and $W: \H^{+} \rightarrow \H^{-} $. Suppose $F$ satisfies:
	\begin{itemize}
		\item $F$ is in the multiplier algebra of $A$;
		\item $F^2-1$ is in $A$.
	\end{itemize}
	Then we define the \emph{index class} $\Ind[F] \in K_0(A)$ of $F$ to be
	\[
	\Ind[F] := 
	\begin{bmatrix}
	(1-VW)^2&V(1-WV) \\
	W(2-VW)(1-VW)&WV(2-WV)
	\end{bmatrix}
	-
	\begin{bmatrix}
	0&0 \\
	0&1
	\end{bmatrix}.
	\]
\end{defn}

\, \\[-.7cm]

Combining Lemma \ref{lem:F multiplier}$\sim$Lemma \ref{lem:projection}, we obtain that for any projection $(p_n) \in \Csq(\H_X)\cap \prod_n\B(\H_n)$ the operator $( (p_n)F_s(p_n) )$ is an odd self-adjoint operator on the graded Hilbert space $\bigoplus_n p_n(L^2 ( [1,\infty), \H_{n,E_n} ) )$ satisfying:
\begin{itemize}
	\item $((p_n)F_s(p_n)) $ is in the multiplier algebra of $(p_n) \Asq(X;E) (p_n)$;
	\item $( (p_n)F_s(p_n) )^2 - (p_n) $ is in $(p_n) \Asq(X;E) (p_n)$.
\end{itemize}
Hence Definition~\ref{defn:index definition} produces an index class in $K_0((p_n) \Asq(X;E) (p_n))$. Composing with the $K_0$-map induced by the inclusion $(p_n) \Asq(X;E) (p_n) \hookrightarrow \Asq(X;E)$, we get an element in $K_0( \Asq(X;E))$, denoted by $\Ind_{F,sq}[(p_n)]$. Analogous to \cite[Lemma~12.3.11]{willett2020higher}, we obtain the following:

\begin{prop}\label{prop:index}
	Through the process above together with a suspension argument, we get well-defined $K_\ast$-maps for $\ast=0,1:$
	\[
	\Ind_{F,sq} : K_\ast \big( \Csq(\H_X)\cap \prod_n\B(\H_n) \big) \rightarrow K_\ast( \Asq(X;E) ),
	\]
which are called the \emph{strongly quasi-local index maps}.
\end{prop}

Finally, we have the follwing result (comparing with Proposition \ref{prop:index map isom. for Roe}). The proof is almost identical to that for \cite[Proposition~12.3.13 and Proposition~12.6.3]{willett2020higher}, hence omitted.

\begin{prop}\label{prop:strong quasi-local index map isom.}
	Given $s\in [1,\infty)$, let $\iota^s:\Asq(X;E)\rightarrow \Csq(\H_{X,E})\cap \prod_n\B(\H_{n,E_n})$ be the evaluation map at $s$ . Then the composition
	\[
	K_* \big( \Csq(\H_X)\cap \prod_n\B(\H_n) \big) \stackrel{\Ind_{F,sq}}{\longrightarrow} K_*(\Asq(X;E)) \stackrel{\iota^s_*}{\longrightarrow} K_*\big( \Csq(\H_{X,E})\cap \prod_n\B(\H_{n,E_n}) \big)
	\]
	is an isomorphism.
\end{prop}

\section{Isomorphisms between twisted algebras in $K$-theory}\label{sec:loc isom}

In this section, we study the $K$-theory of the twisted algebras $A(X;E)$ and $\Asq(X;E)$ defined in Section~\ref{ssec:Twisted Roe and strong quasi-local algebras}. The main result is the following:

\begin{prop}\label{prop:iso. of twisted algebras in $K$-theory}
	The inclusion map from $A(X;E)$ to $\Asq(X;E)$ induces an isomorphism in $K$-theory.
\end{prop}

The proof follows the outline of that in \cite[Section 12.4]{willett2020higher}, and the main ingredient is to use appropriate Mayer-Vietoris arguments for twisted algebras (Proposition~\ref{prop:M_V for twisted algebra}). This allows us to chop the space into easily-handled pieces, on which we prove the required local isomorphisms (Proposition~\ref{prop:local isom. of inclusion in K-theory}).



By saying that $(F_n)_{n \in \NN}$ is a sequence of closed subsets in $(E_n)$, we mean that $F_n$ is a closed subset of $E_n$ for each $n$. Firstly we define the following subalgebras associated to $(F_n)$, which is inspired by \cite[Definition 6.3.5]{willett2020higher}.

\begin{defn}\label{defn:twisted ideal}
For a sequence of closed subsets $(F_n)$ in $(E_n)$, we define $\AA_{sq,(F_n)}(X;E)$ to be the set of elements $(T_{n,s}) \in \AAsq(X;E)$ satisfying the following: for each $n$ and $\varepsilon >0$ there exists $s_{n, \varepsilon }>0$ such that for $s\ge s_{n, \varepsilon }$ we have
	\[
	\supp _{E_n} (T_{n,s}) \subseteq \Nd _\varepsilon (F_n) \times \Nd _\varepsilon (F_n).
	\]
	Denote by $A_{sq,(F_n)}(X;E)$ the norm closure of $\AA_{sq,(F_n)}(X;E)$ in $\Asq(X;E)$. Similarly, we define $A_{(F_n)}(X;E) \subseteq A(X;E)$ in the case of twisted Roe algebra.
\end{defn}

It is easy to see that $A_{sq,(F_n)}(X;E)$ and $A_{(F_n)}(X;E)$ are closed two-side ideals in $\Asq(X;E)$ and $A(X;E)$ respectively. 
Moreover, we have the following: 

\begin{lem}
	Let $(F_n)$ and $(G_n)$ be two sequences of compact subsets in $(E_n)$. Then
	\[
	A_{sq,(F_n)}(X;E) \cap A_{sq,(G_n)}(X;E) = A_{sq,(F_n\cap G_n)}(X;E)
	\]
	and
	\[
	A_{sq,(F_n)}(X;E) + A_{sq,(G_n)}(X;E) = A_{sq,(F_n\cup G_n)}(X;E).
	\]
The same holds for twisted Roe algebras.
\end{lem}

\begin{proof}
We only prove the case of twisted strongly quasi-local algebras, while the Roe algebra case is similar.
The first equation
follows from a $C^*$-algebraic fact that two intersections of ideals coincides with their product together with a basic fact for metric space: For a compact metric space $K$, a closed cover $(C,D)$ of $K$ and $\varepsilon>0$, there exists $\delta>0$ such that $\Nd_\delta(C) \cap \Nd_\delta(D) \subseteq \Nd_\varepsilon(C\cap D)$.
	
For the second, 
we fix $(T_{n,s})\in \AA_{sq,(F_n\cup G_n)}(X;E)$. By definition, for each $n$ there is a strictly increasing sequence $(s_{n,k})_{k\in \NN}$ in $[1,\infty)$ tending to infinity such that for $s\ge s_{n,k}$ we have
	\[
	\supp_{E_n}(T_{n,s}) \subseteq \Nd_{\frac{1}{k+1}}(F_n\cup G_n) \times \Nd_{\frac{1}{k+1}}(F_n \cup G_n).
	\]
	For each $n$, we construct an operator $(W_{n,s})_s$ on $L^2((1,\infty]) \otimes \H_{n,E_n}$ as follows, where $W_{n,s}\in \B(\H_{n,E_n})$. We set:
	\[
	W_{n,s}=
	\begin{cases}
	~\chi_{\Nd_{1}(F_n)}, & \mbox{if~} 1\le s \le s_{n,1};  \\[0.3cm]
	~\frac{s_{n,k+1}-s}{s_{n,k+1} - s_{n,k}}\chi_{\Nd_{\frac{1}{k}}(F_n)} + \frac{s-s_{n,k}}{s_{n,k+1} - s_{n,k}}\chi_{\Nd_{\frac{1}{k+1}}(F_n)}, & \mbox{if~} s_{n,k} \leq s \leq s_{n,k+1}, k\in \NN.
	\end{cases}
	\] 
Then $(W_{n,s})$ is in the multiplier algebra of $\Asq(X;E)$. Now we consider:
	\[
	(T_{n,s}) = (W_{n,s}) (T_{n,s}) + (1-W_{n,s}) (T_{n,s}) (W_{n,s}) +(1-W_{n,s}) (T_{n,s}) (1-W_{n,s}) .
	\]
	It is clear that $(W_{n,s}) (T_{n,s})$ and $(1-W_{n,s}) (T_{n,s}) (W_{n,s})$ are in $A_{sq,(F_n)}(X;E)$. Also note that from the construction above, for each $n$ and $s\ge s_{n,k}$ we have:
	\[
	\supp_{E_n}((1-W_{n,s}) T_{n,s} (1-W_{n,s})) \subseteq \Nd_{\frac{1}{k+1}}( G_n) \times \Nd_{\frac{1}{k+1}}( G_n).
	\]
	Hence we obtain that $A_{sq,(F_n)}(X;E) + A_{sq,(G_n)}(X;E)$ is dense in $A_{sq,(F_n\cup G_n)}(X;E)$, which concludes the proof.
\end{proof}

Consequently, we obtain the following Mayer-Vietoris sequences for twisted algebras:
\begin{prop}\label{prop:M_V for twisted algebra}
	Let $(F_n)$ and $(G_n)$ be two sequences of compact subsets in $(E_n)$. Then we have the following six-term exact sequence:
	\begin{small}
		\[
		\begin{CD}
		K_0(A_{sq,(F_n\cap G_n)}(X;E)) @>>> K_0(A_{sq,(F_n)}(X;E))\oplus K_0(A_{sq,(G_n)}(X;E)) @>>> K_0(A_{sq,(F_n\cup G_n)}(X;E)) \\
		@AAA                               & &                    @VVV \\
		K_1(A_{sq,(F_n\cup G_n)}(X;E)) @<<< K_1(A_{sq,(F_n)}(X;E))\oplus K_1(A_{sq,(G_n)}(X;E)) @<<< K_1(A_{sq,(F_n\cap G_n)}(X;E)).
		\end{CD}
		\]
	\end{small}
The same holds in the case of twisted Roe algebra. Furthermore, we have the following commutative diagram:\\
	\begin{small}
		\centerline{
			\xymatrix{
				\cdots \ar[r] & K_*(A_{(F_n\cap G_n)}(X;E)) \ar[r]  \ar[d] &  K_*(A_{(F_n)}(X;E))\oplus K_*(A_{(G_n)}(X;E)) \ar[d]\ar[r] & K_*(A_{(F_n\cup G_n)}(X;E)) \ar[d]\ar[r]  & \cdots \\
				\cdots \ar[r] & K_*(A_{sq,(F_n\cap G_n)}(X;E)) \ar[r]  & K_*(A_{sq,(F_n)}(X;E))\oplus K_*(A_{sq,(G_n)}(X;E)) \ar[r] & K_*(A_{sq,(F_n\cup G_n)}(X;E)) \ar[r] & \cdots
		}}
	\end{small} 
where the vertical maps are induced by inclusions.
\end{prop}

Proposition \ref{prop:M_V for twisted algebra} allows us to chop the space into small pieces, on which we have the following ``local isomorphism'' result. Recall that a family $\{Y_i\}_{i\in I}$ of subspaces in a metric space $Y$ is \emph{mutually $R$-separated} for some $R>0$ if $d(Y_i,Y_j) >R$ for $i \neq j$.

\begin{prop}\label{prop:local isom. of inclusion in K-theory}
Let $(F_n)$ be a sequence of closed subsets in $(E_n)$ such that $F_n = \bigsqcup_{j=1}^{\infty}F_j^{(n)}$ for a mutually $3$-separated family $\{F_j^{(n)}\}_j$ and there exist $R>0$ and $x^{(n)}_j\in X_n$ such that $F^{(n)}_j \subseteq B(f(x^{(n)}_j);R)$. Then the inclusion map from $A_{(F_n)}(X;E)$ to $A_{sq,(F_n)}(X;E)$ induces an isomorphism in $K$-theory.		
\end{prop}

Before we prove Proposition \ref{prop:local isom. of inclusion in K-theory}, let us use it to finish the proof of Proposition \ref{prop:iso. of twisted algebras in $K$-theory}. To achieve that, we need an extra lemma from \cite[Lemma 12.4.5]{willett2020higher}:
\begin{lem}\label{lem:decomposition}
	For any $s>0$, there exist $M \in \NN$ and decompositions
	\[
	X_n=X_{n,1}\sqcup X_{n,2}\sqcup \cdots \sqcup X_{n,M}, \quad \mbox{for all } n\in \NN,
	\]
such that  the family $\big\{\overline{B(f_n(x);s)}\big\}_{x\in X_{n,i}}$ is mutually $3$-separated for each $n\in\NN$ and $i=1,2,\ldots,M$,.
\end{lem}


\begin{proof}[Proof of Proposition \ref{prop:iso. of twisted algebras in $K$-theory}]
Given $s>0$, let $M\in \NN$ and $\{X_{n,i}\}_{n\in\NN,1\le i\le M}$ be obtained by Lemma~\ref{lem:decomposition}. Setting $W^s_n:=\Nd_s(f_n(X_n))$ and $W_{n,i}^s:=\bigsqcup_{x\in X_{n,i}}\overline{B(f_n(x);s)}$, we have $W^s_n=\bigcup_{i=1}^M W_{n,i}^s$. For each $i$ applying Proposition~\ref{prop:local isom. of inclusion in K-theory} to the sequence of subsets $(W_{n,i}^s)_n$, we obtain that the inclusion map
	\[
	A_{(W_{n,i}^s)}(X;E) \rightarrow A_{sq,(W_{n,i}^s)}(X;E)
	\]
	induces an isomorphism in $K$-theory. Applying the Mayer-Vietoris sequence from Proposition~\ref{prop:M_V for twisted algebra} $(M-1)$-times (and Proposition~\ref{prop:local isom. of inclusion in K-theory} again to deal with the intersection) together with the Five Lemma, we obtain that the inclusion map
	\[
	A_{(W^s_n)}(X;E) \rightarrow A_{sq,(W^s_n)}(X;E)
	\]
	induces an isomorphism in $K$-theory. Finally, note that condition $(3)$ in Definition~\ref{defn:twisted Roe} and condition $(3)$ in Definition~\ref{defn:twisted quasi-local} imply that
	\[
	A(X;E)=\lim_{s\rightarrow \infty}A_{(W^s_n)}(X;E) \quad \mbox{and} \quad A_{sq}(X;E)=\lim_{s\rightarrow \infty}A_{sq,(W^s_n)}(X;E).
	\]
Hence we conclude the proof.
\end{proof}

The rest of this section is devote to the proof of Proposition \ref{prop:local isom. of inclusion in K-theory}. First let us introduce some more notation:

Let $(F_n)$ and $(G_n)$ be sequences of closed subsets in $(E_n)$. We define:
	\[
	\Asq( X;(G_n) ):=(1_{\H_n} \otimes \chi_{G_n})_n \cdot  \Asq(X;E) \cdot (1_{\H_n} \otimes \chi_{G_n})_n
	\]
	and 
	\[
	A_{ sq,(F_n) }( X;(G_n) ):= \Asq( X;(G_n) ) \cap A_{ sq,(F_n) }( X;E ).
	\]
Also define $A( X;(G_n) )$ and $A_{ (F_n) }( X;(G_n) )$ in a similar way. Moreover, given a sequence of subspaces $Z_n \subseteq X_n~ (n\in \NN)$ we define:
	\[
	\Asq( (Z_n) ; (G_n) ):=\big(\chi_{Z_n} \otimes \Id_{\L^2_{E_n}}\big)_n \cdot \Asq( X;(G_n) ) \cdot \big(\chi_{Z_n} \otimes \Id_{\L^2_{E_n}}\big)_n
	\]
	and
	\[
	A_{sq,(F_n) }( (Z_n) ; (G_n) ):=\big(\chi_{Z_n} \otimes \Id_{\L^2_{E_n}}\big)_n \cdot  A_{sq,(F_n) }( X;(G_n) ) \cdot \big(\chi_{Z_n} \otimes \Id_{\L^2_{E_n}}\big)_n.
	\]
Also define $A( (Z_n);(G_n) )$ and $A_{ (F_n) }( (Z_n);(G_n) )$ in a similar way.

Now we move back to the setting of Proposition \ref{prop:local isom. of inclusion in K-theory}. Let $(F_n)$ be a sequence of closed subsets in $(E_n)$ such that $F_n = \bigsqcup_{j=1}^{\infty}F_j^{(n)}$ for a mutually $3$-separated family $\{F_j^{(n)}\}_j$. Taking $G_j^{(n)}=\Nd_{1}(F_j^{(n)})$ for each $j$ and $n$, we define the ``restricted product'':
	\[
	\prod_{j}^{res}A_{sq,(F_j^{(n)}) }( X;(G_j^{(n)}) ) := \big(\prod_{j}A_{sq,(F_j^{(n)}) }( X;(G_j^{(n)}) ) \big) \cap A_{sq,(F_n)}(X;E).
	\]
Similarly, we define $\prod_{j}^{res}A_{ (F_j^{(n)}) }( X;(G_j^{(n)}) )$ in the case of twisted Roe algebra. 

The following lemma is a key step in the proof of Proposition \ref{prop:local isom. of inclusion in K-theory}:

\begin{lem}\label{lem:cluster axiom}
Using the same notation as above, the inclusion
\[
i: \prod_{j}^{res}A_{sq,(F_j^{(n)}) }( X;(G_j^{(n)}) ) \hookrightarrow A_{sq,(F_n)}(X;E)
\]
induces an isomorphism in $K$-theory. The same holds for the twisted Roe algebra case.
\end{lem}

\begin{proof}
We only prove the case of twisted strongly quasi-local algebras, and the Roe case is similar. The proof follows the outline of \cite[Theorem 6.4.20]{willett2020higher}.
	
Consider the following quotient algebras:
	\[
	A_{sq,(F_n),Q}(X;E) := \frac{A_{sq,(F_n)}(X;E)}{A_{sq,(F_n),0}(X;E)} \quad \mbox{and} \quad \prod_{j}^{res,Q}A_{sq,(F_j^{(n)})}( X;(G_j^{(n)}) ):= \frac{\prod_{j}^{res}A_{sq,(F_j^{(n)})}( X;(G_j^{(n)}) )}{\prod_{j}^{res}A_{sq,(F_j^{(n)}),0 }( X;(G_j^{(n)}) )},
	\]
where $A_{sq,(F_n),0}(X;E)$ consists of elements $(T_{n,s})\in A_{sq,(F_n)}(X;E)$ such that $\lim\limits_{s\rightarrow \infty}T_{n,s}=0$ for each $n$, and $A_{sq,(F_j^{(n)}),0 }( X;(G_j^{(n)}) )$ is defined in a simialr way. From a standard Eilenberg Swindle argument (see for example \cite[Lemma 6.4.11]{willett2020higher}), $A_{sq,(F_n),0}(X;E)$ and $\prod_{j}^{res}A_{sq,(F_j^{(n)}),0 }( X;(G_j^{(n)}) )$ both have trivial $K$-theories. Hence the quotient maps 
$$A_{sq,(F_n)}(X;E) \rightarrow A_{sq,(F_n),Q}(X;E) \quad \mbox{and} \quad \prod_{j}^{res}A_{sq,(F_j^{(n)}) }( X;(G_j^{(n)}) )  \rightarrow \prod_{j}^{res,Q}A_{sq,(F_j^{(n)})}( X;(G_j^{(n)}) )$$
induce isomophisms in $K$-theory.

It is clear that the inclusion $i$ induces a $*$-homomorphism:
\[
i_Q:\prod_{j}^{res,Q}A_{sq,(F_j^{(n)})}( X;(G_j^{(n)}) ) \rightarrow A_{sq,(F_n),Q}(X;E).
\]
We also define a map 
	\[
	\gamma :A_{sq,(F_n)}(X;E) \rightarrow \prod_{j}^{res}A_{sq,(F_j^{(n)}) }( X;(G_j^{(n)}) ) \quad \mbox{by} \quad (T_{n,s})\mapsto \prod_{j}(\chi_{G_j^{(n)}}T_{n,s}\chi_{G_j^{(n)}}),
	\]
	which induces a $*$-homomorphism
	\[
	\gamma_Q: A_{sq,(F_n),Q}(X;E) \rightarrow \prod_{j}^{res,Q}A_{sq,(F_j^{(n)})}( X;(G_j^{(n)}) ) .
	\]
We can check that the compositions $i_Q\circ \gamma_Q$ and $\gamma_Q\circ i_Q$ are both the identity maps. Hence $i_Q$ is an isomorphism in $K$-theory, which implies that the inclusion $i$ induces an isomorphism in $K$-theory.
\end{proof}



\begin{proof}[Proof of Proposition \ref{prop:local isom. of inclusion in K-theory}]
We use the same notation as above and define $G_j^{(n)}=\Nd_{1}(F_j^{(n)})$ for each $n\in \NN$ and $j$. Then there is a commutative diagram
	\[
	\begin{CD}
	A_{(F_n)}(X;E) @>>> A_{sq,(F_n)}(X;E) \\
	@AAA       @AAA \\
	\prod_{j}^{res}A_{ (F_j^{(n)}) }( X;(G_j^{(n)}) ) @>>> \prod_{j}^{res}A_{sq,(F_j^{(n)}) }( X;(G_j^{(n)}) )
	\end{CD}
	\]
	where all maps involved are inclusion maps. It follows from Lemma \ref{lem:cluster axiom} that vertical maps induce isomorphisms in $K$-theory. Hence it suffices to show that the bottom horizontal map induces an isomorphism in $K$-theory.
	
Note that conditions (3) in Definition~\ref{defn:twisted Roe} and \ref{defn:twisted quasi-local} imply that
	\[
	\prod_{j}^{res}A_{ (F_j^{(n)}) }( X;(G_j^{(n)}) )=\lim_{m\rightarrow \infty}\prod_{j}^{res}A_{ (F_j^{(n)}) } \big( (B(x^{(n)}_j;m));(G_j^{(n)}) \big)
	\]
and
	\[
	\prod_{j}^{res}A_{ sq,(F_j^{(n)}) }( X;(G_j^{(n)}) )=\lim_{m\rightarrow \infty}\prod_{j}^{res}A_{ sq,(F_j^{(n)}) } \big( (B(x^{(n)}_j;m));(G_j^{(n)}) \big).
	\]
Hence it suffices to show that for each fixed $m$, the inclusion
	\[
	\prod_{j}^{res}A_{ (F_j^{(n)}) } \big( (B(x^{(n)}_j;m));(G_j^{(n)}) \big) \hookrightarrow \prod_{j}^{res}A_{ sq,(F_j^{(n)}) } \big( (B(x^{(n)}_j;m));(G_j^{(n)}) \big)
	\]
	induces an isomorphism in $K$-theory.
	
	Note that the inclusion $\{x^{(n)}_j\} \hookrightarrow B(x^{(n)}_j;m)$ induces a commutative diagram
	\[
	\begin{CD}
	\prod_{j}^{res}A_{ (F_j^{(n)}) } \big( (B(x^{(n)}_j;m));(G_j^{(n)}) \big) @>>> \prod_{j}^{res}A_{ sq,(F_j^{(n)}) } \big( (B(x^{(n)}_j;m));(G_j^{(n)}) \big) \\
	@AAA       @AAA \\
	\prod_{j}^{res}A_{ (F_j^{(n)}) } \big( (\{x^{(n)}_j\});(G_j^{(n)}) \big) @>>> \prod_{j}^{res}A_{ sq,(F_j^{(n)}) } \big( (\{x^{(n)}_j\});(G_j^{(n)}) \big) 
	\end{CD}
	\]
	where the vertical maps are $*$-isomorphisms by standard arguments (see for example Proposition \ref{prop:coarse invariance}). Also note that the bottom horizontal inclusion map $\prod_{j}^{res}A_{ (F_j^{(n)}) }( (\{x^{(n)}_j\});(G_j^{(n)}) ) \hookrightarrow \prod_{j}^{res}A_{ sq,(F_j^{(n)}) }( (\{x^{(n)}_j\});(G_j^{(n)}) )$ is a $*$-isomorphism as well, since 
	conditions $(1)$ and $(3)$ in Definition~\ref{defn:twisted Roe} and~\ref{defn:twisted quasi-local} are equivalent in this case. Hence we conclude the proof.
\end{proof}

\section{Proof of Theorem~\ref{thm:main result}}\label{sec:proof of main thm}

In this final section, we finish the proof of the main result.

\begin{proof}[Proof of Theorem~\ref{thm:main result}]
Consider the following commutative diagram
	\[
	\begin{CD}
	K_*(C^*(\H_X) \cap \prod_n \B(\H_n)) @>>> K_*(A(X;E)) @>>> K_*(C^*(\H_{X,E}) \cap \prod_n \B(\H_{n,E_n})) \\
	@VVV   @VVV   @VVV   \\
	K_*(\Csq(\H_X) \cap \prod_n \B(\H_n)) @>>> K_*(\Asq(X;E)) @>>> K_*(\Csq(\H_{X,E}) \cap \prod_n \B(\H_{n,E_n})),
	\end{CD}
	\]
	where the horizontal maps come from Proposition~\ref{prop:index map isom. for Roe} and Proposition~\ref{prop:strong quasi-local index map isom.} and all vertical maps are induced by inclusions. From Proposition~\ref{prop:index map isom. for Roe} and Proposition~\ref{prop:strong quasi-local index map isom.} again, we know that the compositions of horizontal maps are isomorphisms. The middle vertical map is an isomorphism by Proposition~\ref{prop:iso. of twisted algebras in $K$-theory}, and the left vertical map identifies with the right one due to Proposition~\ref{prop:independent of modules}. Hence the inclusion map
	$$C^*(\H_X) \cap \prod_n \B(\H_n) \hookrightarrow \Csq(\H_X) \cap \prod_n \B(\H_n)$$
	induces an isomorphism in $K$-theory from diagram chasing. Finally combining with Lemma~\ref{lem:reduce to block-diagonal}, we finish the proof.
\end{proof}

\appendix

\section{Proof of Proposition \ref{prop:Psi function}}\label{app:proof2}

This appendix is devoted to the proof of Proposition \ref{prop:Psi function}. We follow the outline of that for \cite[Proposition 12.1.10]{willett2020higher} and use the same notation as in Section \ref{ssec:The Bott-Dirac operators on Euclidean spaces}. 

Define a function $f: \RR \to [-1,1]$ by $f(x)=\frac{x}{\sqrt{1+x^2}}$, $x\in \RR$. Also fix a smooth even function $g:\RR\rightarrow [0,\infty)$ of integral one and having compactly supported Fourier transform. It follows from the proof of \cite[Proposition 12.1.10]{willett2020higher} that given $\varepsilon>0$ there exists $\delta>0$ such that the convolution $\Psi:= f\ast g_\delta$ satisfies condition (1)-(8) in Proposition \ref{prop:Psi function}, where $g_\delta(x):=\frac{1}{\delta}g(\frac{x}{\delta})$ for $x\in \RR$. In the following, we will prove condition (9) and (10) therein. 

Let us recall the following two lemmas, which follow from \cite[Lemma 12.1.6 and 12.1.8]{willett2020higher}.

\begin{lem}\label{lem:appB1}
	For all $s\in[1,\infty)$, $x\in E$ and $t\in \RR$, we have that
	$$f(B_{s,x}-t)=\frac{2}{\pi} \int_{0}^{\infty}(B_{s,x}-t)(1+\lambda^2+(B_{s,x}-t)^2)^{-1} \d\lambda,$$
	where the integral on the right converges in the strong-$\ast$ operator topology.
	
	Moreover for any $s\in[1,\infty)$, $x,y\in E$ and $t\in \RR$, we have that
	\begin{align*}
		f(B_{s,x}-t) - f(B_{s,y}-t)= &c_{x-y}(1+(B_{s,x}-t)^2)^{-\frac{1}{2}} +\frac{2}{\pi} \int_{0}^{\infty} (B_{s,y}-t)(1+\lambda^2+(B_{s,y}-t)^2)^{-1} \\
		&\cdot \big( (B_{s,y}-t) c_{y-x} + c_{y-x} (B_{s,x}-t) \big) (1+\lambda^2+(B_{s,x}-t)^2)^{-1} \d\lambda,
	\end{align*}
	where the integral on the right again converges in the strong-$\ast$ topology.
\end{lem}
\begin{proof}
	The first formula follows from that for any $t\in\RR$, we have the formula
	$$\frac{x-t}{\sqrt{1+(x-t)^2}}= \frac{2}{\pi} \int_0^\infty \frac{x-t}{1+\lambda^2 +(x-t)^2} \d\lambda $$
	and functional calculus. And the second formula follows by easy computations as in the proof of \cite[Lemma 12.1.6]{willett2020higher}.
\end{proof}

\begin{lem}\label{lem:appB2}
	For any $R\ge 0$, $\lambda\in [0,\infty)$, $x\in E$ and $s\ge 2d$, we have that
	$$\|(1+\lambda^2+B_{s,x}^2)^{-\frac{1}{2}}(1-\chi_{x,R})\| \le \big(\frac{1}{2}+\lambda^2+R^2 \big)^{-\frac{1}{4}}.$$
\end{lem}

\begin{proof}[Proof of Proposition \ref{prop:Psi function}(9).]
Given $\varepsilon_1>0$, there exists a compact subset $K\subseteq \RR$ and a function $h: \RR \to [0,\infty)$ of integral one and support in $K$ such that $\|g_\delta - h\|_1<\frac{\varepsilon_1}{4}$.
Setting $\Phi:= f \ast h$, we have:
	$$\|\Psi - \Phi\|_\infty = \|f\ast g_\delta - f\ast h\|_\infty =  \|f\ast (g_\delta - h)\|_\infty \le \|f\|_\infty \cdot \|g_\delta - h\|_1 <\frac{\varepsilon_1}{4},$$
which implies $\|\Phi(B_{s,x}) - \Psi(B_{s,x})\|<\frac{\varepsilon_1}{4}$. Hence it suffices to show that there exists $R_1>0$ such that for all $s\ge 2d$ and $x\in E$, we have
	$$\|(\Phi(B_{s,x})^2-1)(1-\chi_{x,R_1}) \| < \frac{\varepsilon_1}{4}.$$
	
	Now we set $\omega: \RR \to \RR$ by $\omega(x):=\frac{1}{1+x^2}$. For any $R\ge 0$, we have:
	\begin{align*}
		\| (\Phi(B_{s,x})^2-1)(1-\chi_{x,R}) \| 
		&= \big \| \big( (f\ast h)^2 -1 \big) (B_{s,x}) \cdot (1-\chi_{x,R}) \big\| \\
		&= \big\| \big( \frac{(f\ast h)^2 -1}{\omega} \big)(B_{s,x}) \cdot \omega(B_{s,x})(1-\chi_{x,R}) \big\| \\
		&\le \big\| \big( \frac{(f\ast h)^2 -1}{\omega} \big)(B_{s,x}) \big\| \cdot \big\| (1+B_{s,x}^2)^{-\frac{1}{2}} \big\| \cdot \big\| (1+B_{s,x}^2)^{-\frac{1}{2}}(1-\chi_{x,R})\big\| \\
		&\le \big\| \big( \frac{(f\ast h)^2 -1}{\omega} \big)(B_{s,x}) \big\| \cdot \big\| (1+B_{s,x}^2)^{-\frac{1}{2}} \big\| \cdot \big(\frac{1}{2}+R^2 \big)^{-\frac{1}{4}},
	\end{align*}
	where the last inequality comes from Lemma~\ref{lem:appB2} for $\lambda=0$. We claim that the function $\frac{(f\ast h)^2 -1}{\omega}$ is bounded on $\RR$. Indeed, since $h$ has support on $K$ and integral one we have:
	\begin{align*}
		\big( \frac{(f\ast h)^2 -1}{\omega} \big)(x) &= \big(\int_{\RR} f(x-t)h(t) \d t +1 \big) \big(\int_{\RR} f(x-t)h(t) \d t -1 \big) (1+x^2) \\
		&= \big(\int_{K} \big(f(x-t)+1 \big) h(t)\d t \big) \big(\int_{K} \big(f(x-t)-1 \big)h(t)\d t \big) (1+x^2).
	\end{align*} 
Direct calculation shows that
	\begin{align*}
		\big(f(x-t)-1\big)(1+x^2) 
		&= -\frac{x}{\sqrt{1+(x-t)^2}} \cdot \frac{x}{(x-t)+\sqrt{1+(x-t)^2}} \cdot \frac{1+x^2}{x^2},
	\end{align*}
	which is uniformly bounded on $[0,+\infty)$ for $t\in K$. Similarly, $\big(f(x-t)+1\big)(1+x^2)$ is uniformly bounded on $(-\infty,0]$ for $t\in K$. Hence $\frac{(f\ast h)^2 -1}{\omega}$ is bounded on $\RR$.
	
	On the other hand, note that $\| (1+B_{s,x}^2)^{-\frac{1}{2}}\| \le 1$ from functional calculus. Hence we obtain that $\| (\Phi(B_{s,x})^2-1)(1-\chi_{x,R}) \|$ tends to zero as $R$ tends to infinity, which conclude the proof.
\end{proof}

\begin{proof}[Proof of Proposition \ref{prop:Psi function}(10).]
Given $\varepsilon_2>0$, there exists a compact subset $K\subseteq \RR$ and a function $h: \RR \to [0,\infty)$ of integral one and support in $K$ such that $\|g_\delta - h\|_1< \frac{\varepsilon_2}{3}$. Setting $\Phi:= f \ast h$, we have $\|\Phi(B_{s,x}) - \Psi(B_{s,x})\|<\frac{\varepsilon_2}{3}$. Hence it suffices to show that for any $r>0$ there exists $R_2>0$ such that for any $s\ge 2d$ and $x,y\in E$ with $d_E(x,y)\le r$, we have
	$$\|(\Phi(B_{s,x})-\Phi(B_{s,y}))(1-\chi_{x,R_2})\|<\frac{\varepsilon_2}{3}.$$
	
For any $R>0$, we have:
	\begin{align*}
		(\Phi&(B_{s,x})-\Phi(B_{s,y}))(1-\chi_{x,R}) = \big((f\ast h)(B_{s,x})-(f\ast h)(B_{s,y}) \big) (1-\chi_{x,R}) \\
		&= \int_{\RR}\big(f(B_{s,x}-t)-f(B_{s,y}-t)\big)h(t)\d t \cdot (1-\chi_{x,R}).
	\end{align*}
	Combining with Lemma~\ref{lem:appB1}, we have
	\begin{align*}
		&(\Phi(B_{s,x})-\Phi(B_{s,y}))(1-\chi_{x,R}) \\
		&=\int_{\RR} \Bigg( c_{x-y}(1+(B_{s,x}-t)^2)^{-\frac{1}{2}} + \frac{2}{\pi} \int_{0}^{\infty} (B_{s,y}-t)(1+\lambda^2+(B_{s,y}-t)^2)^{-1} \\
		& \quad \cdot \big( (B_{s,y}-t) c_{y-x} + c_{y-x} (B_{s,x}-t) \big) (1+\lambda^2+(B_{s,x}-t)^2)^{-1} \d\lambda \Bigg) h(t)\d t \cdot (1-\chi_{x,R}) \\
		&=\int_{K}c_{x-y}(1+(B_{s,x}-t)^2)^{-\frac{1}{2}}(1-\chi_{x,R})h(t)\d t \\
		&\quad + \frac{2}{\pi} \int_{K}\int_{0}^{\infty} (B_{s,y}-t)(1+\lambda^2+(B_{s,y}-t)^2)^{-1} (B_{s,y}-t) c_{y-x}  (1+\lambda^2+(B_{s,x}-t)^2)^{-1}(1-\chi_{x,R}) \d\lambda h(t)\d t  \\
		&\quad + \frac{2}{\pi} \int_{K}\int_{0}^{\infty} (B_{s,y}-t)(1+\lambda^2+(B_{s,y}-t)^2)^{-1} c_{y-x}(B_{s,x}-t)  (1+\lambda^2+(B_{s,x}-t)^2)^{-1}(1-\chi_{x,R}) \d\lambda h(t)\d t  .
	\end{align*}
	Then it is suffices to show that each of the three terms on the right side tends to zero as $R$ tends to infinity.
	
	For the first term, note that 
the following constant
	\[
	N_1:= \sup_{t\in K, x\in \RR} \frac{\sqrt{1+x^2}}{\sqrt{1+(x-t)^2}}
	\]
is finite since $K$ is compact.  
Hence using Lemma~\ref{lem:appB2} for $\lambda=0$, we obtain
	\begin{align*}
		\Big\| \int_{K}& c_{x-y}(1+(B_{s,x}-t)^2)^{-\frac{1}{2}}(1-\chi_{x,R})h(t)\d t \Big\| \\
		&\le \int_{K} \|c_{x-y}\| \cdot \|(1+(B_{s,x}-t)^2)^{-\frac{1}{2}}(1+B_{s,x}^2)^{\frac{1}{2}}\| \cdot \|(1+B_{s,x}^2)^{-\frac{1}{2}}(1-\chi_{x,R})\|h(t)\d t \\
		&\le r\cdot N_1 \cdot \big( \frac{1}{2}+R^2 \big)^{-\frac{1}{4}},
	\end{align*}
which tends to zero as $R$ tends to infinity.
	
	For the second term, note that 
	\begin{align*}
		&\frac{2}{\pi} \int_{K}\int_{0}^{\infty} (B_{s,y}-t)(1+\lambda^2+(B_{s,y}-t)^2)^{-1} (B_{s,y}-t) c_{y-x}  (1+\lambda^2+(B_{s,x}-t)^2)^{-1}(1-\chi_{x,R}) \d\lambda h(t)\d t \\
		&= \frac{2}{\pi} \int_{K}\int_{0}^{\infty} (B_{s,y}-t)(1+\lambda^2+(B_{s,y}-t)^2)^{-1}(B_{s,y}-t) \cdot c_{y-x} \cdot (1+\lambda^2+(B_{s,x}-t)^2)^{-\frac{1}{2}} \\
		&\quad \cdot (1+\lambda^2+(B_{s,x}-t)^2)^{-\frac{1}{2}} (1+\lambda^2+B_{s,x}^2)^{\frac{1}{2}} \cdot (1+\lambda^2+B_{s,x}^2)^{-\frac{1}{2}} (1-\chi_{x,R}) \d\lambda h(t)\d t
	\end{align*}
	From functional calculus, for any $t\in K$ and $\lambda\in[0,\infty)$ we have
	$$\|(B_{s,y}-t)(1+\lambda^2+(B_{s,y}-t)^2)^{-1}(B_{s,y}-t)\|\le 1$$
	and
	$$\|(1+\lambda^2+(B_{s,x}-t)^2)^{-\frac{1}{2}}\| \le (1+\lambda^2)^{-\frac{1}{2}}.$$
	Also note that the constant
	\[
	N_2:= \sup_{t\in K, x\in \RR, \lambda \in [0,\infty)} \frac{\sqrt{1+\lambda^2+x^2}}{\sqrt{1+\lambda^2 + (x-t)^2}}
	\]
	is finite since $K$ is compact.  
    Hence using Lemma~\ref{lem:appB2}, we obtain
	\begin{align*}
	&\Big\| \frac{2}{\pi} \int_{K}\int_{0}^{\infty} (B_{s,y}-t)(1+\lambda^2+(B_{s,y}-t)^2)^{-1} (B_{s,y}-t) c_{y-x}  (1+\lambda^2+(B_{s,x}-t)^2)^{-1}(1-\chi_{x,R}) \d\lambda h(t)\d t \Big\| \\
	&\le \frac{2}{\pi} \cdot 1 \cdot r \cdot N_2 \cdot \int_{0}^{\infty}(1+\lambda^2)^{-\frac{1}{2}}\big( \frac{1}{2}+\lambda^2 + R^2 \big)^{-\frac{1}{4}}\d\lambda,
	\end{align*}
	which tends to zero as $R$ tends to infinity.
	
	Finally, let us look at the last term. Note that 
	\begin{align*}
		&\frac{2}{\pi} \int_{K}\int_{0}^{\infty} (B_{s,y}-t)(1+\lambda^2+(B_{s,y}-t)^2)^{-1} c_{y-x}(B_{s,x}-t)  (1+\lambda^2+(B_{s,x}-t)^2)^{-1}(1-\chi_{x,R}) \d\lambda h(t)\d t \\
		&=\frac{2}{\pi} \int_{K}\int_{0}^{\infty} (B_{s,y}-t)(1+\lambda^2+(B_{s,y}-t)^2)^{-1} \cdot c_{y-x} \cdot (B_{s,x}-t) (1+\lambda^2+(B_{s,x}-t)^2)^{-\frac{1}{2}} \\
		&\quad \cdot (1+\lambda^2+(B_{s,x}-t)^2)^{-\frac{1}{2}} (1+\lambda^2+B_{s,x}^2)^{\frac{1}{2}} \cdot (1+\lambda^2+B_{s,x}^2)^{-\frac{1}{2}}(1-\chi_{x,R}) \d\lambda h(t)\d t.
	\end{align*}
	It is easy to see that 
	\[
	\sup_{x\in \RR} \big| \frac{x}{1+\lambda^2 + x^2}\big| \leq \frac{1}{2}(1+\lambda^2)^{-\frac{1}{2}}.
	\]
	Hence functional calculus gives that for any $t\in K$,
	$$\|(B_{s,y}-t)(1+\lambda^2+(B_{s,y}-t)^2)^{-1}\| \le \frac{1}{2}(1+\lambda^2)^{-\frac{1}{2}}.$$
	Note also that functional calculus give that for any $t\in K$ and $\lambda\in[0,\infty)$,
	$$\| (B_{s,x}-t) (1+\lambda^2+(B_{s,x}-t)^2)^{-\frac{1}{2}} \| \le 1.$$
	Then using Lemma~\ref{lem:appB2}, we have
	\begin{align*}
	&\Big\| \frac{2}{\pi} \int_{K}\int_{0}^{\infty} (B_{s,y}-t)(1+\lambda^2+(B_{s,y}-t)^2)^{-1} c_{y-x}(B_{s,x}-t)  (1+\lambda^2+(B_{s,x}-t)^2)^{-1}(1-\chi_{x,R}) \d\lambda g(t)\d t \Big\| \\
	&\le \frac{2}{\pi} \cdot r \cdot 1 \cdot N_2 \cdot \int_{0}^{\infty}\frac{1}{2}(1+\lambda^2)^{-\frac{1}{2}}\big( \frac{1}{2}+\lambda^2 + R^2 \big)^{-\frac{1}{4}}\d\lambda,
	\end{align*}
	which tends to zero as $R$ tends to infinity. Hence we conclude the proof.
\end{proof}

\bibliographystyle{amsplain}
\bibliography{bibfileQLCE}

\providecommand{\bysame}{\leavevmode\hbox to3em{\hrulefill}\thinspace}
\providecommand{\MR}{\relax\ifhmode\unskip\space\fi MR }
\providecommand{\MRhref}[2]{%
  \href{http://www.ams.org/mathscinet-getitem?mr=#1}{#2}
}
\providecommand{\href}[2]{#2}
\begin{thebibliography}{10}

\bibitem{engel2015rough}
Alexander Engel, \emph{Rough index theory on spaces of polynomial growth and
  contractibility}, arXiv preprint arXiv:1505.03988 (2015).

\bibitem{guent-tess-yu2012:Borel-stuff}
Erik Guentner, Romain Tessera, and Guoliang Yu, \emph{A notion of geometric
  complexity and its application to topological rigidity}, Invent. Math.
  \textbf{189} (2012), no.~2, 315--357. \MR{2947546}

\bibitem{HRY93}
Nigel Higson, John Roe, and Guoliang Yu, \emph{A coarse {M}ayer-{V}ietoris
  principle}, Math. Proc. Cambridge Philos. Soc. \textbf{114} (1993), no.~1,
  85--97.

\bibitem{structure}
Ana Khukhro, Kang Li, Federico Vigolo, and Jiawen Zhang, \emph{On the structure
  of asymptotic expanders}, arXiv:1910.13320v2 (2019).

\bibitem{Kir34}
Mojzesz Kirszbraun, \emph{{\"U}ber die zusammenziehende und lipschitzsche
  transformationen}, Fundamenta Mathematicae \textbf{22} (1934), no.~1,
  77--108.

\bibitem{lange1985noethericity}
B.~V. Lange and V.~S. Rabinovich, \emph{Noethericity of multidimensional
  discrete convolution operators}, Mat. Zametki \textbf{37} (1985), no.~3,
  407--421, 462. \MR{790430}

\bibitem{intro}
Kang Li, Piotr Nowak, J{\'a}n \v{S}pakula, and Jiawen Zhang, \emph{{Quasi-local
  algebras and asymptotic expanders}}, arXiv:1908.07814, to appear in Groups,
  Geometry and Dynamics.

\bibitem{LWZ2018}
Kang Li, Zhijie Wang, and Jiawen Zhang, \emph{A quasi-local characterisation of
  {$L^p$}-{R}oe algebras}, Journal of Mathematical Analysis and Applications
  \textbf{474} (2019), no.~2, 1213--1237.

\bibitem{roe1988:index-thm-on-open-mfds}
John Roe, \emph{An index theorem on open manifolds. {I}, {II}}, J. Differential
  Geom. \textbf{27} (1988), no.~1, 87--113, 115--136. \MR{918459}

\bibitem{roe1993coarse}
\bysame, \emph{Coarse cohomology and index theory on complete riemannian
  manifolds}, vol. 497, American Mathematical Soc., 1993.

\bibitem{Roe96}
\bysame, \emph{Index theory, coarse geometry, and topology of manifolds}, CBMS
  Regional Conference Series in Mathematics, vol.~90, Published for the
  Conference Board of the Mathematical Sciences, Washington, DC; by the
  American Mathematical Society, Providence, RI, 1996. \MR{1399087}

\bibitem{schick2014:ICM}
Thomas Schick, \emph{The topology of positive scalar curvature}, Proceedings of
  the {I}nternational {C}ongress of {M}athematicians---{S}eoul 2014. {V}ol.
  {II}, Kyung Moon Sa, Seoul, 2014, pp.~1285--1307. \MR{3728662}

\bibitem{ST19}
J\'{a}n {\v S}pakula and Aaron Tikuisis, \emph{Relative {C}ommutant {P}ictures
  of {R}oe {A}lgebras}, Comm. Math. Phys. \textbf{365} (2019), no.~3,
  1019--1048.

\bibitem{SZ20}
J{\'a}n {\v S}pakula and Jiawen Zhang, \emph{Quasi-locality and property {A}},
  Journal of Functional Analysis \textbf{278} (2020), no.~1, 108299.

\bibitem{Wil21}
R.~Willett, \emph{personal communications}, 2021.

\bibitem{Wil09}
Rufus Willett, \emph{Band-dominated operators and the stable {H}igson corona},
  ProQuest LLC, Ann Arbor, MI, 2009, Thesis (Ph.D.)--The Pennsylvania State
  University.

\bibitem{willett2020higher}
Rufus Willett and Guoliang Yu, \emph{Higher index theory}, vol. 189, Cambridge
  University Press, 2020.

\bibitem{Yu95}
Guoliang Yu, \emph{Coarse {B}aum-{C}onnes conjecture}, $K$-Theory \textbf{9}
  (1995), no.~3, 199--221. \MR{1344138}

\bibitem{yu1997:0-in-the-spec-PSC}
\bysame, \emph{Zero-in-the-spectrum conjecture, positive scalar curvature and
  asymptotic dimension}, Invent. Math. \textbf{127} (1997), no.~1, 99--126.
  \MR{1423027}

\bibitem{yu1998:Novikov-for-FAD}
\bysame, \emph{The {N}ovikov conjecture for groups with finite asymptotic
  dimension}, Ann. of Math. (2) \textbf{147} (1998), no.~2, 325--355.
  \MR{1626745}

\bibitem{Yu00}
\bysame, \emph{The coarse {B}aum-{C}onnes conjecture for spaces which admit a
  uniform embedding into {H}ilbert space}, Invent. Math. \textbf{139} (2000),
  no.~1, 201--240.

\end{thebibliography}

\end{document}